\def\version{February 28, 2013}

\documentclass[12pt]{article}

\usepackage[psamsfonts]{amsfonts}
\usepackage{amsmath,amssymb,amsthm}
\usepackage{hyperref}
\usepackage{graphicx}
\usepackage{caption}
\usepackage{subfig}
\usepackage{bbm}

\newcommand\ecu[2]{\begin{equation}\label{#1} #2 \end{equation}}
\newcommand\ecug[1]{\begin{equation*} #1 \end{equation*}}
\newcommand\ecus[2]{\begin{align}\begin{split}\label{#1} #2 \end{split}\end{align}}
\newcommand\ecusg[1]{\begin{align*} #1 \end{align*}}

\newcommand{\ob}[1]{\ensuremath{\left( #1 \right)}}
\newcommand{\sqb}[1]{\ensuremath{\left[ #1 \right]}}
\newcommand{\abs}[1]{\ensuremath{\left| #1 \right|}}
\newcommand{\norma}[1]{\ensuremath{\left|\left| #1 \right|\right|}}
\newcommand{\set}[1]{\ensuremath{\left\{ #1 \right\}}}

\newcommand{\C}{\ensuremath{\mathbb{C}}}
\newcommand{\R}{\ensuremath{\mathbb{R}}}
\newcommand{\Z}{\ensuremath{\mathbb{Z}}}
\newcommand{\Zd}{\ensuremath{\Z^d}}

\newcommand{\1}{\mathbbm{1}}
\newcommand{\D}{\mathrm{d}}

\newcommand{\flechas}{\ensuremath{\leftrightarrow }}

\newcommand{\setE}{\ensuremath{\mathcal E}}

\newcommand{\U}{\ensuremath{\mathcal{U}}}
\newcommand{\V}{\ensuremath{\mathcal{V}}}
\newcommand{\ee}{\ensuremath{\mathrm{e}}}

\newcommand{\z}[1]
{
  \ifnum#1=1
      \ensuremath{z_c^{(t)}}        
  \fi
  \ifnum#1=2
      \ensuremath{z_c^{(a)}}        
  \fi
  \ifnum#1=3
      \ensuremath{z_c}          
  \fi
  \ifnum#1=4
      \ensuremath{z}            
  \fi
}

\newcommand{\g}[1]
{
  \ifnum#1=1
      \ensuremath{g^{(t)}(\z{1})}       
  \fi
  \ifnum#1=2
      \ensuremath{g^{(a)}(\z{2})}       
  \fi
  \ifnum#1=3
      \ensuremath{g(\z{3})}             
  \fi
  \ifnum#1=4
      \ensuremath{g^{(t)}(\z{4})}       
  \fi
  \ifnum#1=5
      \ensuremath{g^{(a)}(\z{4})}       
  \fi
  \ifnum#1=6
      \ensuremath{g(\z{4})}             
  \fi
  \ifnum#1=7
      \ensuremath{g^{(t)}(z_0)}         
  \fi
}

\newcommand{\h}[1]
{
  \ifnum#1=1
      \ensuremath{h_{\z{2}}^{(a)}}      
  \fi
  \ifnum#1=2
      \ensuremath{h_{\z{4}}^{(a)}}      
  \fi
  \ifnum#1=3
      \ensuremath{h_{\z{2}}^{(a)}}      
  \fi
  \ifnum#1=4
      \ensuremath{h_{\z{4}}^{(a)}}      
  \fi
}

\newcommand{\HH}[1]
{
  \ifnum#1=1
      \ensuremath{H_{\z{2}}^{(a)}}      
  \fi
  \ifnum#1=2
      \ensuremath{H_{\z{4}}^{(a)}}      
  \fi
  \ifnum#1=3
      \ensuremath{H_{\z{2}}^{(a)}}      
  \fi
  \ifnum#1=4
      \ensuremath{H_{\z{4}}^{(a)}}      
  \fi
}

\newcommand{\GG}[1]
{
  \ifnum#1=1
      \ensuremath{G^{(t)}_{\z{1}}}      
  \fi
  \ifnum#1=2
      \ensuremath{G^{(a)}_{\z{2}}}      
  \fi
  \ifnum#1=3
      \ensuremath{G_{\z{3}}}            
  \fi
  \ifnum#1=4
      \ensuremath{G^{(t)}_{\z{4}}}      
  \fi
  \ifnum#1=5
      \ensuremath{G^{(a)}_{\z{4}}}      
  \fi
  \ifnum#1=6
      \ensuremath{G_{\z{4}}}            
  \fi
}

\newcommand{\FG}[1]
{
  \ifnum#1=1
      \ensuremath{\hat G^{(t)}_{\z{1}}} 
  \fi
  \ifnum#1=2
      \ensuremath{\hat G^{(a)}_{\z{2}}} 
  \fi
  \ifnum#1=3
      \ensuremath{\hat G_{\z{3}}}   
  \fi
  \ifnum#1=4
      \ensuremath{\hat G^{(t)}_{\z{4}}} 
  \fi
  \ifnum#1=5
      \ensuremath{\hat G^{(a)}_{\z{4}}} 
  \fi
  \ifnum#1=6
      \ensuremath{\hat G_{\z{4}}}   
  \fi
}

\newcommand{\Gk}[2]
{
  \ifnum#1=1
      \ensuremath{ G^{(t,#2)}_{\z{1}}}      
  \fi
  \ifnum#1=2
      \ensuremath{ G^{(a,#2)}_{\z{2}}}      
  \fi
  \ifnum#1=3
      \ensuremath{ G^{(#2)}_{\z{3}}}        
  \fi
  \ifnum#1=4
      \ensuremath{ G^{(t,#2)}_{\z{4}}}      
  \fi
  \ifnum#1=5
      \ensuremath{ G^{(a,#2)}_{\z{4}}}      
  \fi
  \ifnum#1=6
      \ensuremath{ G^{(#2)}_{\z{4}}}        
  \fi
}

\newcommand{\FGk}[2]
{
  \ifnum#1=1
      \ensuremath{\hat G^{(t,#2)}_{\z{1}}}      
  \fi
  \ifnum#1=2
      \ensuremath{ \hat G^{(a,#2)}_{\z{2}}}     
  \fi
  \ifnum#1=3
      \ensuremath{\hat G^{(#2)}_{\z{3}}}        
  \fi
  \ifnum#1=4
      \ensuremath{\hat G^{(t,#2)}_{\z{4}}}      
  \fi
  \ifnum#1=5
      \ensuremath{\hat G^{(a,#2)}_{\z{4}}}      
  \fi
  \ifnum#1=6
      \ensuremath{\hat G^{(#2)}_{\z{4}}}        
  \fi
}

\newcommand{\SQ}[3]
{
  \ifnum#1=1    
      \ensuremath{ S^{(#2,#3)}_{\z{1}}}     
  \fi
  \ifnum#1=2    
      \ensuremath{ S^{(#2,#3)}_{\z{4}}}     
  \fi
  \ifnum#1=3    
      \ensuremath{ S^{(#2,#3)}_{\z{2}}}     
  \fi
  \ifnum#1=4    
      \ensuremath{ S^{(#2,#3)}_{\z{3}}}     
  \fi
}

\newcommand{\PI}[1]
{
  \ifnum#1=1
      \ensuremath{\Pi^{(t)}_{\z{1}}}        
  \fi
  \ifnum#1=2
      \ensuremath{ \Pi^{(a)}_{\z{2}}}       
  \fi
  \ifnum#1=3
      \ensuremath{\Pi_{\z{3}}}          
  \fi
  \ifnum#1=4
      \ensuremath{\Pi^{(t)}_{\z{4}}}        
  \fi
  \ifnum#1=5
      \ensuremath{\Pi^{(a)}_{\z{4}}}        
  \fi
  \ifnum#1=6
      \ensuremath{\Pi_{\z{4}}}          
  \fi
  \ifnum#1=7
      \ensuremath{\tilde \Pi^{(a)}_{\z{4}}} 
  \fi
}

\newcommand{\FPI}[1]
{
  \ifnum#1=1
      \ensuremath{\hat \Pi^{(t)}_{\z{1}}}       
  \fi
  \ifnum#1=2
      \ensuremath{ \hat \Pi^{(a)}_{\z{2}}}      
  \fi
  \ifnum#1=3
      \ensuremath{\hat \Pi_{\z{3}}}             
  \fi
  \ifnum#1=4
      \ensuremath{\hat \Pi^{(t)}_{\z{4}}}       
  \fi
  \ifnum#1=5
      \ensuremath{\hat\Pi^{(a)}_{\z{4}}}        
  \fi
  \ifnum#1=6
      \ensuremath{\hat \Pi_{\z{4}}}         
  \fi
}

\newcommand{\PIN}[2]
{
  \ifnum#1=1
      \ensuremath{\Pi^{(t,#2)}_{\z{1}}}     
  \fi
  \ifnum#1=2
      \ensuremath{ \Pi^{(a,#2)}_{\z{2}}}    
  \fi
  \ifnum#1=3
      \ensuremath{\Pi_{\z{3}}}          
  \fi
  \ifnum#1=4
      \ensuremath{\Pi^{(t,#2)}_{\z{4}}}     
  \fi
  \ifnum#1=5
      \ensuremath{\Pi^{(a,#2)}_{\z{4}}}     
  \fi
  \ifnum#1=6
      \ensuremath{\Pi_{\z{4}}}          
  \fi
}

\newcommand{\FPIN}[2]
{
  \ifnum#1=1
      \ensuremath{\hat \Pi^{(t,#2)}_{\z{1}}}        
  \fi
  \ifnum#1=2
      \ensuremath{\hat \Pi^{(a,#2)}_{\z{2}}}    
  \fi
  \ifnum#1=3
      \ensuremath{\hat \Pi_{\z{3}}}             
  \fi
  \ifnum#1=4
      \ensuremath{\hat \Pi^{(t,#2)}_{\z{4}}}        
  \fi
  \ifnum#1=5
      \ensuremath{\hat \Pi^{(a,#2)}_{\z{4}}}        
  \fi
  \ifnum#1=6
      \ensuremath{\hat \Pi_{\z{4}}}         
  \fi
}

\numberwithin{equation}{section}
\theoremstyle{plain}
\newtheorem{theorem}{Theorem}[section]

\newtheorem{lemma}[theorem]{Lemma}

\begin{document}

\title{
Expansion in high dimension for the
\\
growth constants of lattice trees and lattice animals
}

\author{
Yuri Mej\'ia Miranda\thanks{
Department of Mathematics, University of British Columbia,
Vancouver, BC, Canada V6T 1Z2.  Email: {\tt amie.yuri@gmail.com},
{\tt slade@math.ubc.ca}}
\,
and
Gordon ${\rm Slade}^*$
}

\date{\version}

\maketitle

\begin{abstract}
We compute the first three terms of the $1/d$ expansions
for the growth constants and one-point functions
of nearest-neighbour lattice trees
and lattice (bond) animals on the integer lattice $\Zd$,
with rigorous error estimates.
The proof uses the lace expansion, together with a new expansion
for the one-point functions based on inclusion-exclusion.
\end{abstract}


\section{Main result}
For $d \ge 1$, we consider the integer lattice $\Zd$ as a regular
graph of degree $2d$, with edges consisting of the nearest-neighbour
bonds $\{x,y\}$ with $\|x-y\|_1=1$.
A \emph{lattice animal} is a finite connected subgraph,
and a \emph{lattice tree} is a lattice animal without cycles.
These are fundamental objects in combinatorics and in the
theory of branched polymers
\cite{Jans00}.

We denote the number of lattice animals containing $n$ bonds
and containing the origin of $\Zd$ by $a_n$, and the number
of lattice trees containing $n$ bonds
and containing the origin of $\Zd$ by $t_n$.
Standard subadditivity arguments \cite{Klar67,Klei81} provide the existence
of the $d$-dependent \emph{growth constants} (which we
express in the notation of \cite{GP00})
\begin{equation}
    \lambda_0 = \lim_{n \to \infty} t_n^{1/n},
    \quad\quad
    \lambda_b = \lim_{n \to \infty} a_n^{1/n}.
\end{equation}
A deeper analysis shows that $\lambda_0 = \lim_{n \to \infty} t_{n+1}/t_{n}$
and $\lambda_b = \lim_{n \to \infty} a_{n+1}/a_{n}$ \cite{Madr99}.
The \emph{one-point functions} are the generating functions of
the sequences $a_n$ and $t_n$, namely
\begin{equation}
    g^{(t)}(z)=\sum_{n=0}^\infty t_nz^n \quad \text{and} \quad
    g^{(a)}(z)=\sum_{n=0}^\infty a_nz^n.
\end{equation}
These have radii of convergence $z_c^{(t)}=\lambda_0^{-1}$ and
$z_c^{(a)}=\lambda_b^{-1}$, respectively.
We refer to $z_c^{(t)}$ and
$z_c^{(a)}$ as the \emph{critical points}.
We use superscripts to differentiate between lattice trees and
lattice animals, and we write $z_c$ or $g(z)$ below for
statements that apply to both models.  We use the abbreviation
\begin{equation}
    g_c = g(z_c).
\end{equation}
Also, to make statements simultaneously for lattice trees and
lattice animals, we use the indicator function $\1_{\rm a}$
which takes the value $1$ for the case of lattice animals,
and the value $0$ for the case of lattice trees.

Our main result is the following theorem, which gives detailed
information on the asymptotic behaviour of the critical points
and critical one-point functions as $d \to \infty$.  The notation
$f(d)=o(h(d))$ means $\lim_{d \to \infty} f(d)/h(d)=0$.

\begin{theorem}
\label{thm:1}
For lattice trees or lattice animals, as $d\to\infty$,
\begin{align}
\label{e:zca3}
    z_c    &=\ee^{-1}\left[ \frac{1}{2d}+\frac{\frac{3}{2}}{(2d)^2}
    +\frac{\frac{115}{24} - \1_{\rm a}\frac1{2}\ee^{-1} }{(2d)^3} \right] + o(2d)^{-3},
    \\
\label{e:ga}
    g_c   &= \ee
    \left[ 1 +\frac{\frac{3}{2}}{2d}
    +\frac{ \frac{263}{24} - \1_{\rm a} \ee^{-1} }{(2d)^2}\right]
    +o(2d)^{-2}.
\end{align}
\end{theorem}

Theorem~\ref{thm:1} extends our results in \cite{MS11},
where it was proved that, for both models,
\begin{equation}
\label{e:zcleading}
    z_c = \frac{1}{2d\ee} + o(2d)^{-1},
    \quad\quad\quad g_c = \ee + o(1).
\end{equation}
The leading terms \eqref{e:zcleading} were obtained in \cite{MS11} from
the lace expansion results of \cite{Hara08,HS90b}, together with a
comparison with the mean-field model studied in \cite{BCHS99}.
Our proof of Theorem~\ref{thm:1} provides a different and
self-contained proof of
the asymptotic behaviour of the leading terms, as part of a systematic
development of further terms.

The
lattice trees and lattice animals we are considering are
\emph{bond} clusters. For the closely related models of
\emph{site} trees  and \emph{site} animals, it was proved in
\cite{AB12} and \cite{BBR10} respectively, using very different
methods than ours, that the corresponding growth constants
$\Lambda_0$ and $\Lambda_s$ (in the notation of \cite{GP00}) are
both asymptotic to $2d\ee$ as $d \to \infty$. For related results
for spread-out models of lattice trees and lattice animals, see
\cite{Penr94,MS11}.

The behaviour of $z_c^{(t)}$ and $z_c^{(a)}$
as $d \to \infty$ has been extensively studied in the physics
literature.   For lattice trees, the expansion
\begin{align}
\label{e:zctpredicted}
    z_c^{(t)}   &=\ee^{-1}
    \left[ \frac{1}{2d} + \frac{\frac{3}{2}}{(2d)^2}
    + \frac{\frac{115}{24}}{(2d)^3} + \frac{\frac{309}{16}}{(2d)^4}
    + \frac{\frac{619103}{5760}}{(2d)^5}
    + \frac{\frac{543967}{768}}{(2d)^6} \right]   + \dots.
\end{align}
is equivalent to the expansion given in \cite{GP00}
for $\lambda_0$, but in \cite{GP00} no rigorous estimate for the
error term is obtained.
Similarly, the series
\begin{align}
\label{e:zcapredicted}
    z_c^{(a)}   &=
    \ee^{-1}
    \left[ \frac{1}{2d}+\frac{\frac{3}{2}}{(2d)^2}
    +\frac{\frac{115}{24} - \frac1{2}\ee^{-1} }{(2d)^3}
    + \frac{\frac{309}{16} -2\ee^{-1}}{(2d)^4}
    + \frac{\frac{619103}{5760}- \frac{113}{12}\ee^{-1}}{(2d)^5}
    \right.
    \nonumber \\ & \hspace{70mm} \left.
    + \frac
    {\frac{543967}{768} - \frac{395}{12}\ee^{-1} - \frac{55}{24}\ee^{-2}}
    {(2d)^6}
    \right]   + \dots.
\end{align}
is equivalent to the result of \cite{Harr82,PG95} for $\lambda_b$,
but again no rigorous error estimate was obtained in \cite{Harr82,PG95}.
Equation~\eqref{e:zca3} provides rigorous confirmation of
the first three terms in \eqref{e:zctpredicted}--\eqref{e:zcapredicted},
using completely different methods than \cite{GP00,Harr82,PG95}.

The formulas \eqref{e:zca3}--\eqref{e:zcapredicted} are examples
of $1/d$ expansions.  Such expansions have a long history and
have been developed for several models, in particular for
self-avoiding walk and percolation.
Let $c_n$ denote the number of $n$-step self-avoiding walks starting
at the origin.
For nearest-neighbour self-avoiding walk on $\Zd$, it was proved in
\cite{HS95} that the inverse connective constant
$z_c^{(s)}=[\lim_{n\to\infty} c_n^{1/n}]^{-1}$
has an asymptotic expansion $z_c^{(s)} \sim \sum_{i=1}^\infty m_i (2d)^{-i}$
to all orders,
with all coefficients $m_i$ \emph{integers}.
The first six coefficients had been computed much earlier,
in \cite{FG64}, but without
rigorous control of the error, and these six values were confirmed
with rigorous error estimate in \cite{HS95}.  Subsequently,
seven additional coefficients in the
expansion were computed in \cite{CLS07}.
The values of $m_i$ for $i \le 11$ are positive, whereas $m_{12}$
and $m_{13}$ are negative.
It appears likely that
the series $\sum_i m_i x^i$ has radius of convergence equal
to zero.  It may however be Borel summable, and a partial result in
this direction is given in \cite{Grah10}.  Some related results for
nearest-neighbour bond percolation
on $\Zd$ are obtained in \cite{HS95,HS05,HS06}.
In particular, it is shown in \cite{HS05} that the critical probability
$p_c=p_c(d)$ has an asymptotic expansion $p_c \sim \sum_{i=1}^\infty
q_i (2d)^{-1}$ to all orders, with all $q_i$ \emph{rational}.
The values of $q_1,q_2,q_3$ are computed in \cite{HS95,HS06},
and $q_i$ is given for $i \le 5$ in \cite{GR78} but without
rigorous error estimate.
Results for \emph{spread-out} models of percolation and self-avoiding
walk can be found in \cite{HS05y,Penr93,Penr94}.

An interesting problem which we do not solve in this paper is
to prove existence of asymptotic expansions to all orders for
$z_c^{(t)}$ and $z_c^{(a)}$; we believe that the methods
we develop would be useful for approaching this problem.
An existence proof would then open up the additional problems of
proving that the series have zero radius of convergence but are
Borel summable---the latter problems seem considerably more difficult
than the existence problem.
Also, both the formula \eqref{e:zctpredicted} and the insights in our
proof strongly suggest that there exists an
asymptotic expansion $z_c^{(t)} \sim \ee^{-1} \sum_{i=-1}^\infty r_i(2d)^{-i}$,
with $r_i$ \emph{rational}, but we do not prove this either.
The formula \eqref{e:zca3} does prove that
the coefficients for $z_c^{(a)}$ are not all rational multiples of
$\ee^{-1}$, as was already apparent from the nonrigorous formula
\eqref{e:zcapredicted}.  In our proof, the appearance of the
term $-\frac 12 \ee^{-1}$ in \eqref{e:zca3}
arises due to the contribution from
animals in which the origin lies in a cycle of length $4$, which
of course cannot occur in a lattice tree.
It is in this way that the strict inequality
$z_c^{(a)}<z_c^{(t)}$ \cite{GPSW94} (equivalently $\lambda_0 < \lambda_b$)
first manifests itself in the $1/d$ expansions.

Much has been proved about lattice trees and lattice animals
above the upper critical dimension $d_c=8$, using the lace expansion.
The lace expansion was first adapted to lattice trees and lattice
animals in \cite{HS90b}.  For sufficiently high dimensions,
it has been proved that $t_n \sim A \lambda_0^n n^{-3/2}$
and that the length scale of an $n$-bond lattice tree is typically
of order
$n^{1/4}$ \cite{HS92c}.  Much stronger results relate the scaling
limit of high-dimensional lattice trees to super-Brownian motion
\cite{DS98,Holm08,Slad06}.

The proof of Theorem~\ref{thm:1} relies heavily on the lace expansions
for lattice trees and lattice animals,
and in particular on estimates of \cite{Hara08,HS90b}.
The lace expansions are expansions for the \emph{two-point functions}
\begin{equation}
    G_{z}^{(t)}(x) = \sum_{n=0}^\infty t_n(x) z^n,
    \quad\quad\quad
    G_{z}^{(a)}(x) = \sum_{n=0}^\infty a_n(x) z^n,
\end{equation}
where $t_n(x)$ and $a_n(x)$ respectively denote
the number of $n$-bond lattice trees and
$n$-bond lattice animals containing the
two points $0,x \in \Zd$.
Equivalently,
\begin{equation}
\label{e:2ptfcndef}
     G_{z}(x) = \sum_{C \ni 0,x} z^{|C|},
\end{equation}
where the sum is over lattice trees or lattice animals containing
$0,x$, according to which model is considered, and where $|C|$ denotes
the number of bonds in $C$.

To prove Theorem~\ref{thm:1}, it is not enough just to have an
expansion for the two-point function: an expansion for the
one-point function is needed as well.  This is a
difficulty for lattice trees and lattice animals that does
not occur for self-avoiding walk or percolation.
In this paper, we develop a
new expansion for the one-point function, based on inclusion-exclusion.

The lace expansion and the expansion we present here for the one-point
function have been developed so far only in the context of bond
trees and bond animals.  To apply our approach
to related models, such as site animals or site trees, it would be necessary
to extend the expansions to these models, and also to
extend the estimates of Section~\ref{sec:Fourier} below to these models.

Theorem~\ref{thm:1} first appeared in
the PhD thesis \cite{Meji12}; the proof here has been
reorganised and simplified.

\section{Recursive structure of the proof}
\label{sec:pfstructure}

The \emph{susceptibility} $\chi$ is defined, for lattice trees
or lattice animals, by
\begin{equation}
    \chi(z)= \sum_{x \in \Zd} G_z(x) .
\end{equation}
For $z \in [0,z_c]$,
the lace expansion of \cite{HS90b} expresses $\chi$ in terms of
another function $\hat\Pi_z$ (discussed below in
Section~\ref{sec:LE}) via
\begin{equation}
\label{e:LE}
    \chi(z) =   \frac{g(z) +\hat\Pi_z}{1-2dz(g(z) +\hat\Pi_z)}.
\end{equation}
For $d$ sufficiently large,
the susceptibility has been proven to diverge
at $z_c$ \cite{Hara08,HS90b}, and this is
reflected by the vanishing of the denominator of the right-hand
side of \eqref{e:LE} when $z=z_c$ (see \cite[(1.30)]{Hara08}), namely
\begin{equation}
\label{e:zc}
    1-2dz_c(g_c+ \hat\Pi_{z_c})=0.
\end{equation}
We rewrite \eqref{e:zc} as
\begin{equation}
\label{e:zc-bis}
    z_c  = \frac1{2d}\frac1{g_c+\hat\Pi_{z_c}},
\end{equation}
which expresses $z_c$ in terms of $g_c$ and $\hat\Pi_{z_c}$.

Our main tool in obtaining rigorous error estimates is stated
in Lemma~\ref{lem:DmGn} below.  This lemma applies the infrared bound
of \cite{HS90b}, which is a bound on the Fourier transform of the two-point
function, to obtain estimates on certain convolutions of the two-point function.
Using Lemma~\ref{lem:DmGn}, we prove the following expansions for $G_{z_c}(s)$
and for $\hat\Pi_{z_c}$, where $s\in \Zd$ is a neighbour of the origin.
Recall that $\1_{\rm a}$ equals $1$ for lattice animals and equals
$0$ for lattice trees.

\begin{theorem}\label{thm:G_2}
Let $s\in \Zd$ be a neighbour of the origin.
For lattice trees or lattice animals,
\ecu{e:G0x_2}{
    \GG{3}(s)
    = \ee
    \left[ \frac{1}{2d} + \frac{\frac72}{(2d)^2}\right] + o (2d)^{-2} .
}
\end{theorem}

\begin{theorem}
\label{thm:Pi-2terms}
For lattice trees or lattice animals,
\begin{align}
\label{e:Pia-2terms}
    \hat\Pi_{z_c}
    &=
    \ee\left[
    -\frac{3}{2d} - \frac{ \frac{27}2 - \1_{\rm a} \frac32 \ee^{-1}}{(2d)^2}
    \right]
    +  o(2d)^{-2}.
\end{align}
\end{theorem}

Our method of proof follows
a recursive
procedure in which the calculation of the terms in the
expansion for $z_c$ is intertwined with the computation of the
terms in the expansions for $G_{z_c}(s)$, $\hat\Pi_{z_c}$ and $g_c$.
A key ingredient is the new expansion for the one-point function
developed in Section~\ref{sec:onept}.
Although \eqref{e:zcleading} has been proved already in \cite{MS11},
we give a different proof as the initial step in the recursion.
Our proof here is conceptually simpler and more direct
than that of \cite{MS11},
and also serves as a good introduction to the systematic computation
of higher order terms.
Our starting point consists of the estimates (valid for large $d$)
\begin{equation}
\label{e:step0}
   1 \le g_c \le 4,
    \quad\quad
    2dz_c g_c = 1 + o(1).
\end{equation}
The first of these bounds is proved in \cite{HS90b} for both lattice
trees and lattice
animals (the lower bound is trivial),
and the second is a consequence of \eqref{e:zc}
together with the estimate
$\hat\Pi_{z_c} = O(2d)^{-1}$ proved in \cite{HS90b}.  We comment in
more detail on the previously known bounds on $\hat\Pi_z$ in Section~\ref{sec:LE}
below.
It is an immediate consequence of \eqref{e:step0} that for large $d$ we have
$2dz_cg_c \in [\frac 12, 2]$, and hence
\begin{equation}
\label{e:zcbd0}
   \frac{ \frac{1}{8}}{2d} \le z_c \le \frac{2}{2d}.
\end{equation}

\begin{figure}[!h]
\centering
\includegraphics[scale=.4]{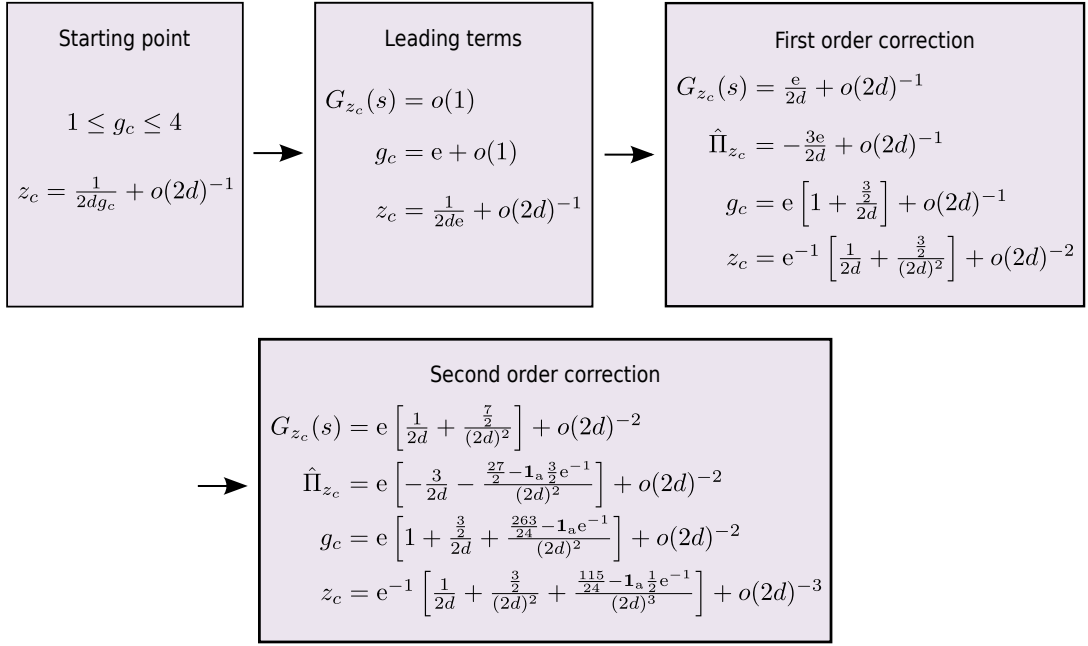}
\caption{Flow of the proof of Theorem~\ref{thm:1}.
The steps represented by the three arrows
are implemented in Sections~\ref{sec:term1},
\ref{sec:term2} and \ref{sec:term3}, respectively.
}
\label{flowchart}
\end{figure}

Our procedure consists of the three steps depicted in Figure~\ref{flowchart}.
In Section~\ref{sec:term1},
we first apply Lemma~\ref{lem:DmGn} to prove that
$G_{z_c}(s) = o(1)$, as a
very preliminary version of Theorem~\ref{thm:G_2}.
With \eqref{e:step0},
this permits us to apply the simplest version of our new
expansion for the one-point function to improve
\eqref{e:step0}--\eqref{e:zcbd0} to $g_c = \ee +o(1)$ and $z_c = (2d\ee)^{-1}
+o(2d)^{-1}$,
yielding \eqref{e:zcleading}.
Then in Section~\ref{sec:term2},
we apply \eqref{e:zcleading} to
compute the first terms on the right-hand sides
of \eqref{e:G0x_2}--\eqref{e:Pia-2terms}, then use the result
of that computation together with the expansion for the
one-point function to compute the second term of \eqref{e:ga},
and then from \eqref{e:zc-bis} obtain the second term of
\eqref{e:zca3}.
In Section~\ref{sec:term3},
we repeat the process,
obtaining an additional term for $\hat\Pi_{z_c}$, then an
additional term for $g_c$.
Once we have proved Theorem~\ref{thm:Pi-2terms}
and \eqref{e:ga},
the expansion \eqref{e:zca3}
follows immediately by substitution into \eqref{e:zc-bis}.
Due to the algorithmic nature of the procedure, there
is no reason in principle why further terms could not be
computed with further effort.
The results in Sections~\ref{sec:onept} and \ref{sec:term1}--\ref{sec:term3}
heavily rely on several technical estimates which we collect and
prove in Section~\ref{sec:cie}.

\section{Expansion for one-point function}
\label{sec:onept}

In this section, we develop a new expansion for the one-point
functions of lattice trees and lattice animals, simultaneously.
The expansion may be considered as a systematic use of
inclusion-exclusion to compare with the mean-field model
of lattice trees of \cite{BCHS99}, which is
based on the Galton--Watson branching process with critical
Poisson offspring distribution.

\subsection{Estimate for the one-point function}

We begin by stating the one result from Section~\ref{sec:onept},
in Theorem~\ref{thm:gstart} below,
that will be used later in the proof of Theorem~\ref{thm:1}.
The proof of Theorem~\ref{thm:gstart} uses
\emph{only} the starting bounds
\eqref{e:step0}, together with the important
Lemma~\ref{lem:DmGn} which is used to bound errors.

In the case of $g^{(a)}(z)$, it is convenient to separate
the sum over lattice animals depending on whether
the origin is contained in a cycle or not, which we
denote by $0\in {\rm cycle}$ and $0\not \in {\rm cycle}$, respectively.
For the former, we define
\begin{equation}
    g_\circ (z) = \1_{\rm a} \sum_{A\ni 0: 0\in {\rm cycle}}z^{|A|}.
\end{equation}
Then we obtain,
for either model,
\begin{equation}
\label{e:gLA}
    g (z)
    = \sum_{C\ni 0}z^{|C|}
    =\sum_{C\ni 0: 0\not\in {\rm cycle}}z^{|C|}
    + g_\circ (z),
\end{equation}
where the clusters $C$ are lattice trees or lattice animals depending
on which model we consider.  We will expand the first term on the
right-hand side of \eqref{e:gLA}, but do not expand $g_\circ(z)$.

We introduce the notion of a \emph{planted} tree or
animal as one which contains the origin as a vertex of degree 1.
An important role will be played by the generating function
\begin{equation}
\label{e:rdef}
    r(z)=\sum_{S\ni s} z^{|S|},
    \quad\quad  r_c = r(z_c),
\end{equation}
for clusters planted via the bond $\{0,s\}$ with $s$ a specific
neighbour of the origin
(by symmetry $r(z)$ does not depend on the choice of $s$).
We emphasise that in \eqref{e:rdef} we are abusing notation by
writing $S \ni s$ to denote
that the \emph{bond} $\{0,s\}$ is contained in
the planted cluster $S$; we will continue to use
this notational convention.
The generating function $r$ is related to the one- and two-point functions
by the identity
\begin{equation}
\label{e:rgG}
    r(z) = zg(z) - z G_z(s) .
\end{equation}
To see this, we use the definition of $r$ and inclusion-exclusion to
write
\begin{equation}
    r(z) = z \sum_{C \ni s : C \not\ni 0}z^{|C|}
    = z \sum_{C \ni s }z^{|C|}
    - z \sum_{C \ni s , 0}z^{|C|},
\end{equation}
and observe that the resulting right-hand side is identical to the
right-hand side of \eqref{e:rgG}.

At the critical value $z_c$, we can use \eqref{e:zc} to replace $g$ by
$\hat\Pi$ in \eqref{e:rgG}, and obtain
\begin{equation}
\label{e:rPiG}
    2dr_c = 1 - 2dz_c \hat\Pi_{z_c} - 2dz_c G_{z_c}(s).
\end{equation}
The identity \eqref{e:rPiG} will be useful in conjunction with the
following theorem.

\begin{theorem}
\label{thm:gstart}
For lattice trees or lattice animals,
\begin{align}
\label{e:gstart}
    g_c & =  \ee^{2dr_c}
    \left[
    1 - \frac 12 (2d)r_c^2 + \frac{1}{8}(2d)^2 r_c^4 - \frac{\frac 76}{(2d)^2}  \right]
    + g_\circ(z_c)
    +o(2d)^{-2}  .
\end{align}
\end{theorem}

The proof of Theorem~\ref{thm:gstart} will be discussed at the
end of Section~\ref{sec:onept}.
It is based on the expansion for $g$ which
we discuss next.
The remainder of Section~\ref{sec:onept} is
needed only for the proof of Theorem~\ref{thm:gstart}.

\subsection{Expansion for the one-point function}

The one-point function for trees, and for animals
in which the origin does not belong to a cycle,
have the following similar structure.
A tree $T$, or an animal $A$ for which the origin is not in
a cycle, consists either of the single vertex $0$, or of some
number $m \in \{1,\ldots, 2d\}$ of planted clusters
$S_i$ which intersect
pairwise only at the origin.  This is depicted in
Figure~\ref{fig:plantedclusters}.

\begin{figure}[!h]
\centering
\includegraphics[scale=.4]{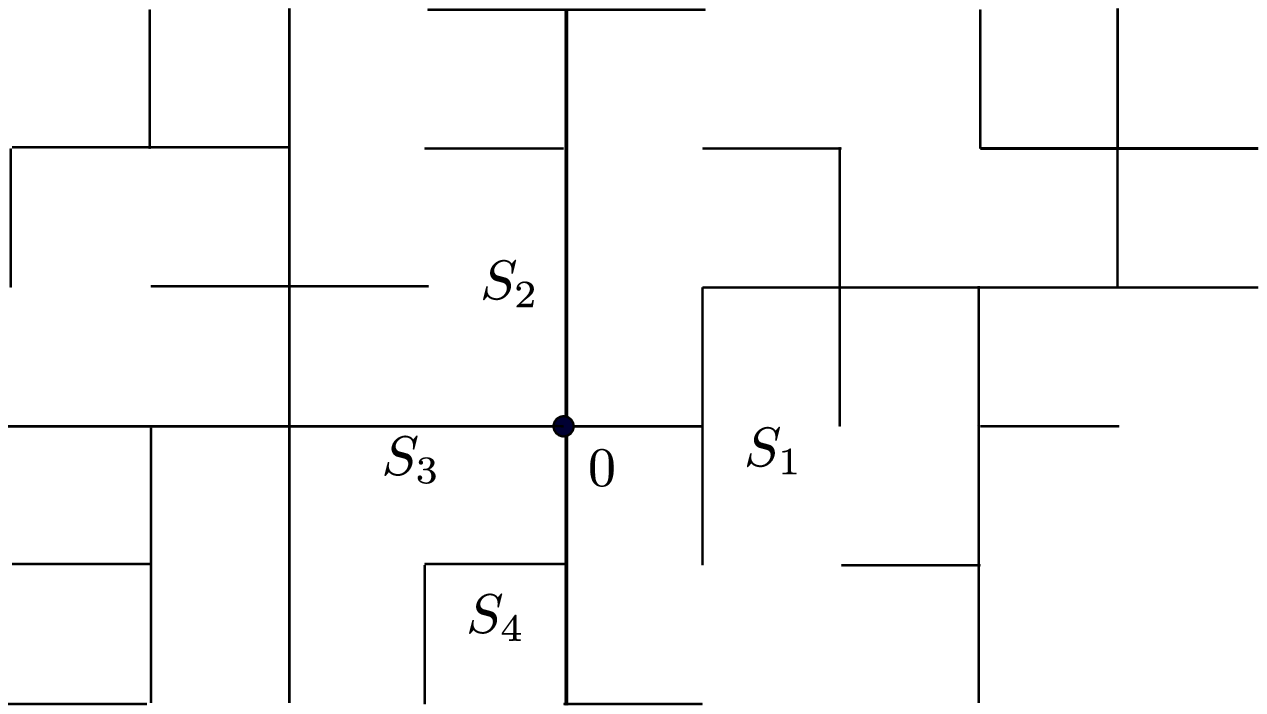}
\quad
\includegraphics[scale=.4]{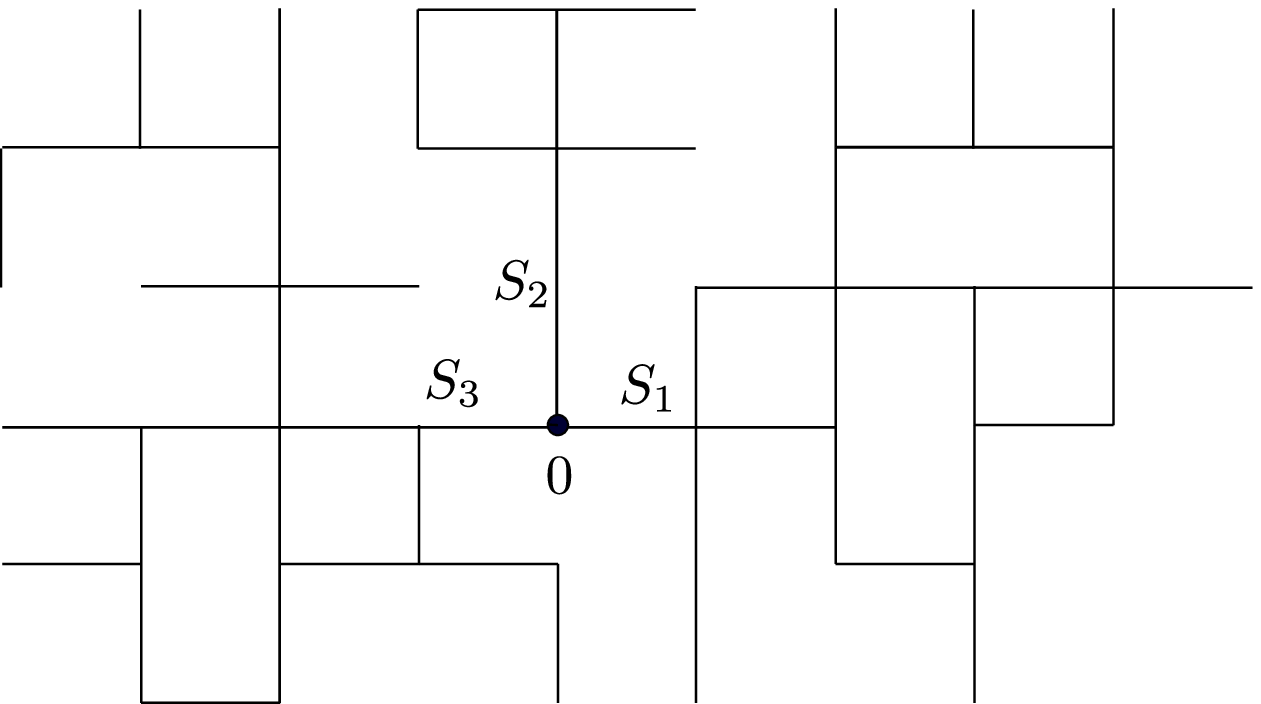}
\caption{Decomposition into planted clusters
$S_i$, for a lattice tree and for a lattice animal
with $0 \not \in \text{cycle}$.
\label{fig:plantedclusters}
}
\end{figure}

Given $0\leq i<j\leq m$ and a set $\vec
S=\{S_1,\dots,S_m\}$ of $m$ planted clusters,
we define
\begin{equation}
    {\cal V}_{ij}({\vec S})=
        \begin{cases}
        -1
        & \text{if $S_i$ and $S_j$ share a common vertex other than $0$}
        \\
        0
        &\text{if $S_i$ and $S_j$ share no common vertex  other than $0$.}
        \end{cases}
\end{equation}
Let ${\cal E} =\{e_1,e_2,\ldots,e_{2d}\}$ consist of the $2d$
nearest neighbours of the origin ordered such that
$e_i=(0,\dots,1,0,\dots,0)$, where the 1 is located at the
$i$-th coordinate for $1\leq i\leq d$, and $e_i=-e_{i-d}$
for $d+1\leq i\leq 2d$.
Then we can rewrite the one-point function as
\begin{align}
\label{e:LEgz}
    g(z)
    &=\sum_{m=0}^{\infty}\frac{1}{m!} \sum_{s_1,\dots,s_m\in {\cal E}}
    \sum_{S_1\ni s_1} z^{|S_1|} \cdots
    \sum_{S_m\ni s_m} z^{|S_m|}
    \prod_{1\leq i<j\leq m} (1+{\cal V}_{ij}) + g_\circ (z).
\end{align}
The factor
$(1+{\cal V}_{ij})$ ensures that $S_i$ and $S_j$ do not
intersect each other except at the origin;
in particular, this excludes the possibility
that $s_i=s_j$.
It also ensures that the sum over $m$ in \eqref{e:LEgz} is actually
a finite sum, since the terms vanish for $m>2d$.

It follows easily by induction on $n \ge 0$ that
\begin{equation}
\label{e:prodxi}
    \prod_{1 \le a \le n} (1 + x_a )
    =   1   +   \sum_{1 \le a \le n}  x_a \prod_{a <b \le n} ( 1 + x_b ).
\end{equation}
Throughout the paper, an empty product equals $1$ and an empty sum
equals $0$.
Iteration of \eqref{e:prodxi} gives
\begin{align}
\label{e:prodxi3}
    \prod_{1 \le a \le n}( 1 + x_a )
    &=  1   +   \sum_{1 \le a \le n}  x_a
    + \sum_{1 \le a < b \le n} x_a x_b
    + \sum_{1 \le a < b < c  \le n}x_ax_bx_c
    \nonumber \\ & \quad
    + \sum_{1 \le a < b < c < d \le n}x_ax_bx_cx_d
    \prod_{d<e \le n}( 1 + x_e ).
\end{align}
We apply \eqref{e:prodxi3} to the product
$\prod_{1\leq i<j\leq m}(1+{\cal V}_{ij})$
in \eqref{e:LEgz},
with the lexicographic order on the indices $(i,j)$.
To facilitate this, for $m \ge 2$ we define
\begin{equation}
\label{e:setAij}
    A_{ij}= A_{ij}(m) =
    \{(i,l): j < l \le m\} \cup
    \{ (k,l) :
    i<k <  l \le m \};
\end{equation}
thus $A_{ij}$ consists of the indices that are lexicographically larger
than $(i,j)$.
Then \eqref{e:prodxi3} gives
\begin{equation}
    \prod_{1\leq i<j\leq m} (1+{\cal V}_{ij})
    =
    {\cal J}_m^{(0)}
    -
    {\cal J}_m^{(1)}
    +
    {\cal J}_m^{(2)}
    -
    {\cal J}_m^{(3)}
    +
    \tilde{\cal J}_m^{(4)},
\end{equation}
where
\begin{align}
\label{e:Def_Jcal0}
    {\cal J}_m^{(0)} & =1,
    \\
\label{e:Def_Jcal1}
    {\cal J}_m^{(1)}
    & =
    \sum_{1\leq i<j\leq m}(-{\cal V}_{ij}),
    \\
    {\cal J}_m^{(2)}
    & =
    \sum_{1\le i<j \le m}
    \sum_{ (k,l)\in A_{ij} } {\cal V}_{ij}{\cal V}_{kl},
    \\
\label{e:Def_Jcal3}
    {\cal J}_m^{(3)}
    & =
    \sum_{1\le i<j \le m}
    \sum_{ (k,l)\in A_{ij} }
    \sum_{(p,q)\in A_{kl}} (-{\cal V}_{ij}{\cal V}_{kl} \V_{pq}),
    \\
    \tilde{\cal J}_m^{(4)}
    & =
    \sum_{1\le i<j \le m}
    \sum_{ (k,l)\in A_{ij} }
    \sum_{(p,q)\in A_{kl}}
    \sum_{(r,s)\in A_{pq}}
    {\cal V}_{ij}{\cal V}_{kl} \V_{pq}\V_{rs}{\cal I}_{rs}
    ,
\end{align}
with ${\cal I}_{rs}  =   \prod_{(t,u)\in A_{rs}} ( 1 + {\cal V}_{tu} )$.
This leads to the expansion
\begin{equation}
\label{e:gexp}
    g(z)
    = \Gamma^{(0)}(z) - \Gamma^{(1)}(z)+
    \Gamma^{(2)}(z) - \Gamma^{(3)}(z)+
    \tilde\Gamma^{(4)}(z) + g_\circ(z),
\end{equation}
where
\begin{align}
\label{e:B}
    \Gamma^{(i)}(z)  &
    = \sum_{m=0}^{\infty}\frac{1}{m!}
    \sum_{s_1,\dots,s_m\in {\cal E}}
    \sum_{S_1\ni s_1} z^{|S_1|} \cdots
    \sum_{S_m\ni s_m} z^{|S_m|}
    {\cal J}_m^{(i)} \quad\quad (i=0,1,2,3)
    ,
    \\
\label{e:Gam4til}
    \tilde\Gamma^{(4)}(z) & =
    \sum_{m=2}^{\infty}\frac{1}{m!} \sum_{s_1,\dots,s_m\in {\cal E}}
    \sum_{S_1\ni s_1} z^{|S_1|} \cdots  \sum_{S_m\ni s_m}
    z^{|S_m|}
    \tilde {\cal J}_m^{(4)}.
\end{align}
Note that $i$ in $\Gamma^{(i)}$ counts the
number of factors of $\V$ in each term, and the remainder term
$\tilde\Gamma^{(4)}$ also contains the factor ${\cal I}_{rs}$.
This last factor could be expanded further, and the
process continued indefinitely, but for the proof
of Theorem~\ref{thm:1} the expansion \eqref{e:gexp} suffices.

\subsection{Identities and estimates for the one-point function}

In this section, we first
prove identities needed for the analysis of the $\Gamma^{(i)}$.
These are then used,
together with estimates whose proofs are deferred to Section~\ref{sec:cie},
to provide an expansion for $g_c$ in terms of $r_c$.

The term $\Gamma^{(0)}(z)$ can be immediately computed.  Indeed,
by its definition in \eqref{e:B} and \eqref{e:Def_Jcal0}, and by \eqref{e:rdef},
\begin{align}
\label{e:Ar}
    \Gamma^{(0)}(z)  &
    = \sum_{m=0}^{\infty}\frac{ 1 }{ m! } (2d)^m r(z)^m
    =
    \ee^{2d  r(z)}
    .
\end{align}
The term $\Gamma^{(1)}$ is also straightforward, as we show below.
For the analysis of $\Gamma^{(2)}(z)$ and $\Gamma^{(3)}(z)$,
it will be useful to decompose
according to the cardinality of the label sets
$\{i,j,k,l\}$ and $\{i,j,k,l,p,q\}$
(respectively in ${\cal J}_m^{(2)}$ and ${\cal J}_m^{(3)}$)
and we write $\Gamma^{(m,n)}$
for the contribution to $\Gamma^{(m)}$ arising
from label sets of cardinality $n$.  Thus, for $m=2,3$, $\Gamma^{(m)}
=\sum_n \Gamma^{(m,n)}$,
where $m$ counts the number of $\V$ factors and $n$ counts the
cardinality of the label set.
In particular, when $m=2$ we have the two possibilities $n=3,4$,
while for $m=3$ the possibilities are $n=3,4,5,6$.
As we discuss in more detail below, $\Gamma^{(3,n)}$ is an error
term for $n=4,5,6$, as is $\tilde\Gamma^{(4)}$.
For Theorem~\ref{thm:1}, we will need an accurate calculation of
$\Gamma^{(2,3)}(z)$, $\Gamma^{(2,4)}(z)$ and $\Gamma^{(3,3)}(z)$.
To obtain convenient expressions
for these important terms, we make the definitions
\begin{align}
\label{e:Zdef}
    Z^{(1)}(z) &=
    \sum_{s_1,s_2\in \setE}
    \sum_{S_1\ni s_1} \sum_{S_2\ni s_2}
    z^{\abs{S_1} + \abs{S_2}}  (-{\cal V}_{12}),
\\
\label{e:Zstar}
    Z^{(2)}(z)
    &=
    \sum_{s_1,s_2,s_3 \in \setE}
    \sum_{S_1\ni s_1}\sum_{S_2 \ni s_2}\sum_{S_3\ni s_3}
    \z{4}^{\abs{S_1}+ \abs{S_2} +\abs{S_3}} \V_{12}\V_{13},
    \\
\label{e:Zstar2}
    Z^{(3)}(z)
    & =
    \sum_{s_1,s_2,s_3 \in \setE}
    \sum_{S_1\ni s_1}\sum_{S_2 \ni s_2}\sum_{S_3\ni s_3}
    \z{4}^{\abs{S_1}+ \abs{S_2} +\abs{S_3}} (-\V_{12}\V_{13}\V_{23})
    .
\end{align}

\begin{lemma}
\label{lem:Gams}
The following identities hold:
\begin{align}
\label{e:BAZ}
    \Gamma^{(1)}(\z{4})  &=
    \frac {1}{2!} \Gamma^{(0)}(z)  Z^{(1)}(z),
    \\
\label{e:C3-3}
    \Gamma^{(2,3)}(z) & =
    \frac{3}{3!} \Gamma^{(0)} (z)
    Z^{(2)}(z),
    \\
\label{e:C4decomp}
    \Gamma^{(2,4)}(z) & =
    \frac{3}{4!} \Gamma^{(0)} (z)
    Z^{(1)}(z)^2
    ,
    \\
\label{e:C23decomp}
    \Gamma^{(3,3)}(z) & =
    \frac{1}{3!} \Gamma^{(0)}(z )
    Z^{(3)}(z).
\end{align}
\end{lemma}

\begin{proof}
For $\Gamma^{(1)}(z)$,
we interchange the sums over $s_1,\dots,s_m\in \setE$
and $1\leq i<j\leq m$ which arise by substitution of \eqref{e:Def_Jcal1}
into \eqref{e:B} to obtain
\begin{align}
    \Gamma^{(1)}(\z{4})
    &=\sum_{m=2}^{\infty}\frac{1}{m!} (2dr(z))^{m-2}
        \sum_{1\leq i<j\leq m} \hspace{.2cm} \sum_{s_i,s_j\in \setE}
        \sum_{S_i\ni s_i}
          \z{4}^{\abs{S_i}} \sum_{S_j\ni s_j} \z{4}^{\abs{S_j}} (-{\cal V}_{ij})
    \nonumber \\
    &= \sum_{m=2}^{\infty}\frac{1}{m!}(2dr(z))^{m-2}
    \binom{m}{2}
    Z^{(1)}(z)
     \nonumber \\
    &= \frac 12 \Gamma^{(0)}(z)  Z^{(1)}(z),
\end{align}
where we used \eqref{e:Ar} in the last step.

For $\Gamma^{(2,3)}$, the condition $\abs{\{i,j,k,l\}}=3$ is satisfied when
$k=i$, $k=j$ or $l=j$. In all cases, we choose three
labels from a set of $m$ and order them; this order
automatically determines which one corresponds to $i$, $j$, $k$ and $l$.
Hence, the number of options for the labels is $3{m\choose 3}$.
Using symmetry, we obtain
\begin{align}
\label{e:C3-3-pf}
    \Gamma^{(2,3)}(z) & =
    \sum_{m=3}^\infty \frac{1}{m!} \ob{ 2dr(z) }^{m-3}
    3 \binom{m}{3}
    \sum_{s_1,s_2,s_3 \in \setE}
    \sum_{S_1\ni s_1}\sum_{S_2 \ni s_2}\sum_{S_3\ni s_3}
    \z{4}^{\abs{S_1}+ \abs{S_2} +\abs{S_3}} \V_{12}\V_{13}
    \nonumber \\
    & =
    \frac{3}{3!} \Gamma^{(0)}(z )
    Z^{(2)}(z).
\end{align}

For the case $\Gamma^{(2,4)}$, the labels
$i,j,k,l$ are distinct.
To determine the number of possibilities for the labels we chose
four labels from a set of $m$ and order them. Then $i$ is the
smallest by definition, $j$ has the remaining 3 options,
and once $j$ is determined, so are $k$ and $l$. Hence, there are
$3\binom{m}{4}$ possibilities.
By interchanging sums and using symmetry,
we obtain
\begin{align}
\label{e:C4decomp-pf}
    \Gamma^{(2,4)}(z) & =
    \sum_{m=4}^\infty \frac{1}{m!} \ob{ 2dr(\z{4}) }^{m-4}
    3 \binom{m}{4}
    \sum_{s_1,s_2,s_3,s_4 \in \setE}
    \sum_{S_1\ni s_1}\sum_{S_2 \ni s_2}\sum_{S_3\ni s_3}\sum_{S_4\ni s_4}
    \z{4}^{\abs{S_1}+ \abs{S_2} +\abs{S_3}+\abs{S_4}} \V_{12}\V_{34}
    \nonumber \\ & =
    \frac{3}{4!}  \Gamma^{(0)}(z)
    Z^{(1)}(z)^2
    .
\end{align}

For $\Gamma^{(3,3)}$, it must be the case that
$i<j<l$, $k=i$, $p=j$ and $q=l$.
Thus the number of possibilities for the labels is given by choosing
three labels from a set of $m$ and ordering them in this way.
By interchanging sums and using symmetry,
we obtain
\begin{align}
    \Gamma^{(3,3)}(z) & =
    \sum_{m=3}^\infty \frac{1}{m!} \ob{ 2dr(z) }^{m-3}
    \binom{m}{3}
    \sum_{s_1,s_2,s_3 \in \setE}
    \sum_{S_1\ni s_1}\sum_{S_2 \ni s_2}\sum_{S_3\ni s_3}
    \z{4}^{\abs{S_1}+ \abs{S_2} +\abs{S_3}} (-\V_{12}\V_{13}\V_{23})
    \nonumber \\ & =
    \frac{1}{3!} \Gamma^{(0)}(z )
    Z^{(3)}(z)
    .
\end{align}

This completes the proof.
\end{proof}

Now we can prove Theorem~\ref{thm:gstart}, using estimates
from Section~\ref{sec:Zestimates}.
The estimates we require are that
\begin{align}
\label{e:Zbd3}
    Z^{(1)}_c
    &=
    2dr_c^2 + \frac{3}{(2d)^2} + o(2d)^{-2}
    ,
    \quad
    Z^{(2)}_c
    =
    \frac{ 1 }{(2d)^2}
    + o(2d)^{-2}
    ,
    \quad
    Z^{(3)}_c
    =
    \frac{ 1}{(2d)^2}
    + o(2d)^{-2}
\end{align}
(proved in Lemma~\ref{lem:ZZZ}), and that
the terms $\Gamma^{(3,n)}(z_c)$
($n=4,5,6$) and $\tilde\Gamma^{(4)}(z_c)$ are all $O(2d)^{-3}$
(proved in Lemma~\ref{lem:Gamma34}).  The proofs of
Lemmas~\ref{lem:ZZZ}--\ref{lem:Gamma34} depend
only on the starting bounds \eqref{e:step0},
together with Lemma~\ref{lem:DmGn} which gives error estimates.

\begin{proof}[Proof of Theorem~\ref{thm:gstart}.]
We substitute the identities of Lemma~\ref{lem:Gams} into \eqref{e:gexp},
and apply the results of Lemmas~\ref{lem:ZZZ}--\ref{lem:Gamma34}
mentioned above
(together with $r_c \le z_c g_c=O(2d)^{-1}$ by \eqref{e:step0}),
to obtain
\begin{align}
\label{e:gstart-pf}
    g_c & =  \ee^{2dr_c}
    \left[
    1 - \frac{1}{2!}Z^{(1)}_c +
    \left(\frac{3}{3!} Z^{(2)}_c+  \frac{3}{4!} ( Z^{(1)}_c)^2 \right)
    -
    \frac{1}{3!} Z^{(3)}_c \right]
    + g_\circ(z_c) +o(2d)^{-2}
    \nonumber
    \\
    & =  \ee^{2dr_c}
    \left[
    1 - \frac 12 (2d)r_c^2 + \frac{1}{8}(2d)^2 r_c^4 - \frac{\frac 76}{(2d)^2}  \right]
    + g_\circ(z_c)
    +o(2d)^{-2}   ,
\end{align}
and the proof is complete.
\end{proof}

\section{Lace expansion}
\label{sec:LE}

We recall some fundamental facts about the lace expansion for lattice
trees and lattice animals from \cite{HS90b} (see also
\cite{Hara08,Slad06}).

\subsection{Lace expansion for lattice trees}

A lattice tree containing $0,x$, which contributes to the two-point
function $G_z(x)=\sum_{T \ni 0,x}z^{|T|}$ of
\eqref{e:2ptfcndef}, can be decomposed into a unique path joining
$0$ and $x$, which we call the \emph{backbone},
together with the disjoint collection of subtrees consisting of
the connected components that remain after the bonds in the backbone
(but not the vertices) are removed.  We refer to the subtrees
(which may consist of a single vertex) as
\emph{ribs}.  The definitions should be clear from Figure~\ref{fig:tree_br}.

\begin{figure}[!h]
 \centering
 \includegraphics[scale=.5]{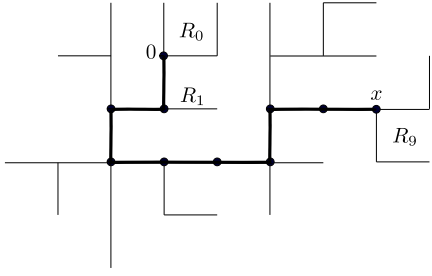}
 \caption[Decomposition of a lattice tree]
 {Decomposition of a lattice tree $T$ into the
  backbone from 0 to $x$ (bold) and the ribs
  $\vec R=\{R_0,\dots,R_{9} \}$.}
 \label{fig:tree_br}
\end{figure}

By definition, the ribs are mutually avoiding.  However,
in high dimensions, if this avoidance restriction were relaxed then
intersections between ribs should be in some
sense still rare.  The lace expansion is a way of making this
vague intuition
precise, via a systematic use of inclusion-exclusion.
To describe the basic idea, we need the following definitions.

Let $D:\Zd\to\R$ denote the one-step transition probability function
for simple random walk on $\Zd$, i.e.,
\ecus{eq:D}{
    D(x)    &=
          \begin{cases}
              (2d)^{-1}&\text{if }\|x\|_{1}=1,\\
              0 &\text{otherwise.}
          \end{cases}
}
The {\it convolution} of absolutely summable functions
$f:\Zd\to\R$ and $h:\Zd\to\R$ is given by
\ecu{eq:conv}{
    (f*h)(x)=\sum_{y\in\Zd}f(y)h(x-y).
}

If it were the case that the rib $R_0$ were permitted to intersect
the remaining ribs, then the two-point function
$G_z^{(t)}(x)$ (for $x \neq 0$)
would be given by the convolution
\begin{equation}
\label{e:2ptinc}
    g^{(t)}(z) (2dzD*G_z^{(t)})(x)
    = g^{(t)}(z) \sum_{y \in \Zd} 2dzD(y)G_z^{(t)}(x-y),
\end{equation}
where the factor $g^{(t)}(z)$ captures the rib at the origin, $y$ is the
location of the next vertex after $0$ along the backbone,
and $G_z^{(t)}(x-y)$ captures the backbone from $y$ to $x$ together with
its ribs.  Compared to the two-point function, \eqref{e:2ptinc}
permits disallowed intersections and thus
includes too much.  In fact, it provides the basis of the
mean-field model introduced in \cite{DS97} and further
studied in \cite{BCHS99,Slad06}. The lace expansion corrects
the overcounting in \eqref{e:2ptinc} with the help of the function
$\Pi_z : \Zd \to \R$ which appears in the \emph{identity}
\begin{equation}
    G_z^{(t)}(x) = \delta_{0,x} g^{(t)}(z)
    + \Pi_z^{(t)}(x) + g^{(t)}(z) (2dz D*G_z^{(t)})(x)
    + (\Pi_z^{(t)} * 2dzD*G_z^{(t)})(x).
\end{equation}
In \cite{HS90b}, an
expansion for $\hat\Pi_z^{(t)} = \sum_{x\in \Zd} \Pi_z^{(t)}(x)$ is given,
of the form
\begin{equation}
    \hat\Pi_z^{(t)} = \sum_{N=1}^\infty (-1)^N \hat\Pi_z^{(t,N)}.
\end{equation}
It is known  (see \cite{Hara08})
that there is a $c>0$ such that for all $N \ge 1$ and all $z \in [0,z_c]$,
\begin{equation}
\label{e:HSPibds}
    0 \le \hat\Pi_{z}^{(t,N)} \le c^N d^{-N}
\end{equation}
and this implies that the only terms that can
contribute to \eqref{e:Pia-2terms} for lattice trees are those with
$N=1,2$.  We define these terms next.

We define $\U_{ij}(\vec R)$ by
\ecu{eq:Uij}{
    \U_{ij}(\vec R)=
        \begin{cases} -1
        & \text{if ribs $R_i$ and $R_j$ share a common vertex}
        \\
        0 &\text{if ribs $R_i$ and $R_j$ share no common vertex.}
        \end{cases}
}
Let $\mathcal W(x)$ denote the set of simple random walk paths $\omega$
from $0$ to $x$, i.e., sequences $x_0=0,x_1,\ldots, x_n=x$ with
$\|x_{i+1}-x_i\|_1=1$ for all $i$, for any length $n =|\omega| \ge 0$.
The function $\PIN{4}{1}(x)$ is defined by
\ecus{Pi_uno}{
    \PIN{4}{1}(x)
    &=  \sum_{\omega\in\mathcal W(x) :|\omega|\geq 1}
    \z{4}^{|\omega|}
    \sum_{R_0\ni\omega(0)} \z{4}^{|R_0|}
    \cdots
    \sum_{R_{|\omega|}\ni x} \z{4}^{|R_{|\omega|}|}
    (-\U_{0 |\omega|})
    \prod_{0\leq i<j\leq |\omega|\atop(i,j)\neq(0,|\omega|)}
    \ob{1+\U_{ij}}.
}
For a nonzero contribution,
the factor $\U_{0 |\omega|}$ forces the first and last ribs to
intersect, while the final product disallows all other
intersection among the ribs.
The function $\PIN{4}{2}(x)$ is defined by
\ecu{eq:Pi2}{
    \PIN{4}{2}(x)
    = \sum_{\omega\in\mathcal W(x): |\omega|\ge 2}
    \z{4}^{|\omega|}
    \sum_{R_0\ni\omega(0)} \z{4}^{|R_0|}
    \cdots
    \sum_{R_{|\omega|}\ni x} \z{4}^{|R_{|\omega|}|}
    \sum_{L \in\mathcal L ^{(2)}[0,|\omega|]} \,
    \prod_{ij\in L}\U_{ij}\prod_{i'j'\in \mathcal C(L)}
    \ob{1+ \U_{i'j'} }  ,
}
where the set $\mathcal L^{(2)}[0,|\omega|]$ of
(2-edge) \emph{laces}
is given by:
\begin{equation}
\label{e:Lcal2}
    \mathcal L^{(2)}[0,n]
    =
    \big\{ \set{0j,jn} : 0<j<n \big\}
    \cup
    \big\{\set{0j,in} : 0<i<j<n \big\},
\end{equation}
and where the set ${\mathcal C}(L)$ \emph{compatible} with
$L \in {\mathcal L}^{(2)}[0,n]$ is given:
\\
(i) for $L=\{0j,jn\}$, by all pairs $kl$
with $0 \le k < l \le n$ except $0l$ with $l>j$ and $kn$ with
$k<j$;
\\
(ii) for $L=\{0j,in\}$ with $i<j$,
by all pairs $kl$ excepting both $0l$ with $l>j$ and $kn$ with
$k<i$.
\\
For more details, see \cite{HS90b} or \cite{Hara08,Slad06}.

\subsection{Lace expansion for lattice animals}

The two-point function
$\GG{5}(x)=\sum_{A\ni 0,x}\z{4}^{|A|}$ for lattice animals
was defined in \eqref{e:2ptfcndef}
as the sum over lattice animals that contain both vertices
0 and $x$. An animal $A$ with this characteristic contains a path
connecting $0$ to $x$; however, unlike the lattice tree case, this
path is not necessarily unique. To deal with this we use
the following definitions.

Let $A$ be an animal containing the vertices $x$ and $y$.
We say that
$A$ has a \emph{double connection from $x$ to $y$} if
  there are two bond-disjoint self avoiding walks in $A$ between
  $x$ and $y$ (the walks may share a common vertex, but not a common bond),
  or if $x=y$. The set of all animals having a double connection
  between $x$ and $y$ is denoted by $\mathcal D_{x,y}$.
A bond $\set{x,y}$ in $A$ is \emph{pivotal} for the connection
 from $x$ to $y$, if its removal would disconnect the animal into
 two connected components with $x$ contained in one of them and
 $y$ in the other.

An animal $A\ni x,y$ that is not an element of $\mathcal D_{x,y}$
has at least one pivotal bond for the connection from $x$ to $y$.
To establish an order among these
edges, we define the {\it first} pivotal bond to be the unique bond
for which there is a double connection between $x$ and one of the
endpoints of this bond. This endpoint
is the {\it first} endpoint of the first pivotal bond. To determine
the  {\it second} pivotal bond, the role of $x$ is played by the
second endpoint of the first pivotal
bond, and so on.

For a lattice animal $A$ that contains $x$ and $y$, the {\it backbone}
is the ordered set of
oriented pivotal bonds for the connection from $x$ to $y$.
The backbone is not necessarily connected. The {\it ribs} are
the connected components that remain after the bonds in the
backbone (but not the vertices) are removed
from $A$. By definition, the ribs are doubly connected between
the corresponding backbone vertices, and are mutually avoiding.
See Figure~\ref{fig:animal_br} for an example.
\begin{figure}[!h]
 \centering
 \includegraphics[scale=.5]{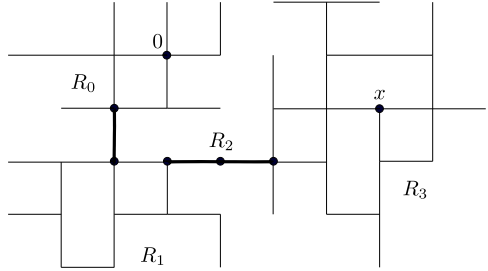}
 \caption[Decomposition of a lattice animal]
 {Decomposition of a lattice animal $A$ into the backbone
 from 0 to $x$ (bold), and the ribs
 $\vec R=\set{R_0,R_1,R_2,R_{3} }$.
 The rib $R_2$ consists only of the vertex in the backbone.}
 \label{fig:animal_br}
\end{figure}

Let $B$ be an arbitrary finite ordered set of directed bonds
\ecug{
    B   =   \ob{ \ob{u_1,v_1},\dots, \ob{u_{|B|},v_{|B|}} },
}
and let $v_0=0$ and $u_{|B|+1}=x$.
Then we can regard the two-point function as a sum over
the backbone $B$ and mutually nonintersecting ribs
$\vec R=\set{R_0,\dots,R_{|B|}}$.
It is shown in \cite{HS90b} how to apply inclusion exclusion to
obtain an identity
\begin{equation}
    G_z^{(a)}(x) = \delta_{0,x} g^{(a)}(z)
    + \Pi_z^{(a)}(x) + g^{(a)}(z)(2dz D*G_z^{(a)})(x)
    + (\Pi_z^{(a)} * 2dzD*G_z^{(a)})(x),
\end{equation}
with $\Pi_z^{(a)}$ given by the alternating series
\begin{equation}
    \hat\Pi_z^{(a)} = \sum_{N=0}^\infty (-1)^N \hat\Pi_z^{(a,N)}.
\end{equation}
It is known  (see \cite{Hara08})
that there is a $c>0$ such that for all $N \ge 0$ and all $z \in [0,z_c]$,
\begin{equation}
\label{e:HSPibda}
    0 \le \hat\Pi_{z}^{(a,N)} \le c^N d^{-(N \vee 1)}
\end{equation}
and this implies that the only terms that can
contribute to \eqref{e:Pia-2terms} for lattice animals are those with
$N=0,1,2$.

The following explicit formulas are obtained in \cite{HS90b}.
First,
\ecu{eq:Pi0}{
    \PIN{5}{0}(x)
    =  \ob{ 1- \delta_{0,x} }\sum_{R\in \mathcal D_{0,x}}\z{4}^{|R|}.
}
With $\U_{ij}(\vec R)$ as in \eqref{eq:Uij} but for the new
notion of ribs $\vec R$,
\ecus{LA_Pi_uno}{
    \PIN{5}{1}(x)
    &=  \sum_{B:|B|\ge1} \z{4}^{|B|}
    \sqb{ \prod_{k=0}^{|B|}\quad
    \sum_{R_k\in\mathcal D_{v_k,u_{k+1}}}\z{4}^{|R_k|} }
    (-\U_{0,|B|})
    \prod_{0\leq i<j\leq |B|\atop(i,j)\neq(0,|B|)}\ob{1+\U_{ij}},
}
with $v_0=0$ and $u_{|B|+1}=x$.
The factor $\U_{0,|B|}$ in the previous expression forces an
intersection between the first and last ribs,
and the last product forbids all other rib intersections.
Finally,
\ecu{eq:Pia2}{
    \PIN{5}{2}(x)
    = \sum_{B:|B|\ge1} \z{4}^{|B|}
    \sqb{ \prod_{k=0}^{|B|}\quad
    \sum_{R_k\in\mathcal D_{v_k,u_{k+1}}}\z{4}^{|R_k|} }
    \sum_{L \in\mathcal L ^{(2)}[0,|B|]} \,
    \prod_{ij\in L}\U_{ij}\prod_{i'j'\in \mathcal C(L)}
    \ob{1+ \U_{i'j'} }   ,
}
with $\mathcal L ^{(2)}$ and $ \mathcal C(L)$ as defined around
\eqref{e:Lcal2}.

\section{Fourier estimates}
\label{sec:Fourier}

In this section,
we formulate
an essential ingredient for the error estimates in
Theorem~\ref{thm:1}, in Lemma~\ref{lem:DmGn} below.
The proof is based on the Fourier transform.

The {\it Fourier transform} of an absolutely summable
function $f:\Zd\to\C$ is defined by
\ecu{eq:FT}{
    \hat f(k)=\sum_{x\in\Zd}f(x)\ee^{ik\cdot x},
}
where $k\in [-\pi,\pi]^d$ and $k\cdot x=\sum_{j=1}^d k_jx_j$.
For example, the transition probability $D$ of \eqref{eq:D}
has Fourier transform $\hat D(k) = d^{-1}\sum_{j=1}^d \cos k_j$.
The {\it inverse Fourier transform}, which recovers $f$ from $\hat f$,
is given by
\ecu{eq:IFT}{
    f(x)=\int_{[-\pi,\pi]^d} \hat f(k) \ee^{-ik\cdot x}\frac{\D k}{(2\pi)^d}.
}
Recall that the convolution of the functions $f$ and $g$ was
defined in \eqref{eq:conv}. We denote by $f^{*l}$ the
convolution of $l$ factors of $f$,
i.e.,
\ecug{
    f^{*l}(x)=\underbrace{(f*f*\cdots *f)}_{l}(x).
}
The Fourier transform of a convolution
is the product of Fourier transforms:
$\widehat{f*g}=\hat f\hat g$.

In this notation, $D^{*l}(x)$ is the $l$-step transition probability
that simple random walk travels from $0$ to $x$ in $l$ steps.
We take $f=D^{*2m}$ and $x=0$ in \eqref{eq:IFT} to obtain
\begin{equation}
    D^{*2m}(0)
    =
    \int_{[-\pi,\pi]^d} \hat D(k)^{2m} \frac{\D k}{(2\pi)^d}.
\end{equation}
A proof of the elementary fact that $D^{*2m}(0) \le C_m(2d)^{-m}$ for
some constant $C_m$ (uniformly in $d$) can be found in
\cite[(3.12)]{HS06}.
Therefore,
\ecu{eq:D_bound}{
    \int_{[-\pi,\pi]^d} \hat D(k)^{2m} \frac{\D k}{(2\pi)^d}
    \leq    \frac{C_m}{(2d)^m}.
}

The {\it infrared bound} for nearest-neighbour lattice trees and
lattice animals, given in \cite[(1.25)]{Hara08}, states
that for dimensions $d\ge d_0$ (for some sufficiently large $d_0$),
there is a positive constant $c$ independent of $\z{4}$ and $d$,
such that for $0\le \z{4}\le z_c$,
\ecu{eq:infrared}{
    0 \leq \FG{6}(k) \leq \frac{cd}{|k|^2},
}
where $|k|=\ob{k_1^2+\dots+k_d^2}^{1/2}$.
The definition of $ \FG{6}(k)$ requires some care when $z=z_c$,
because $G_{z_c}(x)$ is not summable.  Nevertheless it is possible
to define $\hat{G}_{z_c}(k)$ in a natural way such that its inverse Fourier
transform is $G_{z_c}(x)$.  The subtleties associated with this
point are discussed in \cite[Appendix~A]{Hara08}.

Let $i$ be a non-negative integer and let $C$ be a cluster
(a tree or an animal) containing the vertices $x$ and $y$. We denote by
\ecu{eq:xy_k}{
    \{ x\underset{i}{\flechas}y \}
}
the event that there exists a self-avoiding path in $C$,
of length at least $i$, connecting  $x$ and $y$.
We define
the two-point function
for clusters in which $x$ is connected to $y$ by a path
of length at least $i$ by
\ecu{eq:Gk}{
    \Gk{6}{i} (x) =\sum_{ C\ni 0,x \; : \;
    0\underset{i}{\flechas} x } \z{4}^{|C|}.
}
Then
\ecug{
    \GG{6}(x) =\Gk{6}{0}(x)=\g{6}\delta_{0,x}+\Gk{6}{1}(x),
}
since for $x=0$ the two-point function $\GG{6}(x)$ reduces
to the one-point function $\g{6}$, and for $x\not = 0$ a path
connecting these two vertices requires at least one step.

For integers  $m,n \ge 1$, and vertices $x,y \in \Zd$,
we define
\ecu{eq:Smn}{
    \SQ{2}{m}{n}(x)   =\sum_{ i_1+\dots +i_n=m }
    (\Gk{6}{i_1}*\cdots *\Gk{6}{i_n})(x)
    ,
}
where the sum is over nonnegative integers $i_1,\dots, i_n$.
Let
\ecu{eq:Smn_def}{
    \SQ{2}{m}{n}    =\sup_{x \in \Zd} \SQ{2}{m}{n}(x).
}
The statement and proof of the following lemma are closely
related to \cite[Lemma~3.1]{HS06}.

\begin{lemma}\label{lem:DmGn}
Let $m$ and $n$ be non-negative integers and
let $d> \max\set{d_0, 4n}$.  There is a
constant $C_{m,n}$, whose value depends only on $m$ and $n$,
such that
\begin{align}
\label{eq:Smn_bound}
    S^{(m,n)}_{z_c}
    &\leq \frac{C_{m,n}}{(2d)^{m/2}}.
\end{align}
\end{lemma}

\begin{proof}
We first prove that there is a constant $K_{m,n}$ such that
\begin{align}
\label{eq:DG_bound}
      \sup_{x} \left( D^{*m}* G_{z_c}^{*n} \right)(x)
      &\leq \frac{K_{m,n}}{(2d)^{m/2}}.
\end{align}
Using the inverse Fourier transform \eqref{eq:IFT}
and $\widehat{f*g}=\hat f\hat g$, we have
\ecusg{
    (D^{*m}* G_{z_c}^{*n}) (x)
    &= \int_{[-\pi,\pi]^d} \hat D(k)^{m} \hat G_{z_c}(k)^{n}
    \ee^{-ik\cdot x}\frac{\D k}{(2\pi)^d}.
}
By the Cauchy--Schwarz inequality,
\ecus{eq:D_G}{
    (D^{*m}* G_{z_c}^{*n}) (x)
    &\leq \ob{ \int_{[-\pi,\pi]^d} \hat D(k)^{2m}
    \frac{\D k}{(2\pi)^d} }^{1/2}
    \ob{ \int_{[-\pi,\pi]^d} \hat G_{z_c}(k)^{2n} \frac{\D k}{(2\pi)^d} }^{1/2}.
}
Then \eqref{eq:D_bound}
gives \eqref{eq:DG_bound}, once we show that the second factor on
the right-hand side of \eqref{eq:D_G} is bounded uniformly in large $d$.
By \eqref{eq:infrared}, it suffices to verify that the integral
\begin{equation}
    I_{d,n} =
     \int_{[-\pi,\pi]^d}\frac{d^{2n}}{|k|^{4n}}
  \frac{\D k}{(2\pi)^d},
\end{equation}
which is finite for $d>4n$, is monotone nonincreasing in $d$.

This monotonicity has been encountered many times previously
in the literature (e.g., \cite{HS06}),
and can be proved as follows.
For $A>0$ and $j>0$, a change of variables in the integral leads to
\begin{equation}
    \frac1{A^j} =\frac1{\Gamma(j)}\int_0^\infty u^{j-1}\ee^{-uA}\D u.
\end{equation}
We apply this identity with $A=d^{-1}|k|^2$ and $j=2n$,
and then use Fubini's theorem to obtain
\begin{align}
    I_{d,n}
    &=\frac1{\Gamma(2n)}\int_{[-\pi,\pi]^d}
    \int_0^\infty u^{2n-1}\ee^{-u|k|^2/d}\D u
    \frac{\D k}{(2\pi)^d}
    \nonumber \\
    &=\frac1{\Gamma(2n)} \int_0^\infty u^{2n-1}
    \ob{ \int_{-\pi}^\pi \ee^{-ut^2 /d}\frac{\D t}{2\pi} }^d \D u
    =\frac1{\Gamma(2n)} \int_0^\infty u^{2n-1}
    \| f_u\|_{1/d} \; \D u,
\end{align}
where $f_u(t)= \ee^{-ut^2}$ and $\|f\|_p
= (\int_{-\pi}^\pi f(t)^p \, \D t/2\pi)^{1/p}$.
Since $\D t/2\pi$ is a probability measure on $[-\pi,\pi]$,
\begin{equation}
    \norma{f}_{1/(d+1)} \leq \norma{f}_{1/d}.
\end{equation}
Therefore, as required, $I_{d+1,n} \le I_{d,n}$, and the proof
of \eqref{eq:DG_bound} is complete.

Turning now to \eqref{eq:Smn_bound}, we first consider the
case of lattice trees.  In
\eqref{eq:Gk}, if we neglect the self-avoidance restriction
among the first $i$ steps in the path connecting $x$ and $y$,
and treat the first $i$ ribs as independent
of each other and of the subtree that comes after the
$i^{\rm th}$ step, we
obtain the upper bound
\ecu{eq:Gk_bound_DG}{
    G_{z}^{(i)}(x)  \leq  (2d\z{4}\g{6})^i   (D^{*i}*G_{z})(x).
}
For the case of lattice animals, the same bound is plausible
and indeed also holds; this
can be seen using a small modification in the proof of \cite[Lemma~2.1]{HS90b}.
With the definition of $\SQ{2}{m}{n}(x)$ in \eqref{eq:Smn}, this
implies that for either model
\begin{align}
    \SQ{2}{m}{n}(x) &=\sum_{i_1+\dots +i_n=m}
    (G_{z}^{(i_1)}*\cdots *G_{z}^{(i_n)})(x)
    \le \tilde C_{m,n} \ob{ 2d\z{4}\g{6}}^m (D^{*m}*G_{z}^{*n})(x),
\end{align}
where $\tilde C_{m,n}$ is the number of terms in the sum
(its exact value is unimportant).
By \eqref{e:step0}, $2d\z{3}g_c  \leq  2$ for $d$ large enough.
Together with \eqref{eq:DG_bound}, this implies that
\ecusg{
    S^{(m,n)}_{z_c}(x) &
    \leq \tilde C_{m,n} 2^{m}  \frac{K_{m,n}}{(2d)^{m/2}},
}
and the proof is complete.
\end{proof}

\section{First term}
\label{sec:term1}

In this section, we apply \eqref{e:step0}
to compute the leading behaviour \eqref{e:zcleading}
for $g_c$ and $z_c$.  This provides an alternate approach to that used
in \cite{MS11} to reach the same conclusion, and makes our proof
of Theorem~\ref{thm:1} more self-contained.
The following lemma provides some preliminary bounds.

\begin{lemma}
For $s$ a neighbour of the origin,
\begin{align}
\label{e:Gstart}
    G_{z_c}(s) & = o(1),
\\
\label{e:rcbd}
    2dr_c & = 1+o(1),
\\
\label{e:gcirc0}
    g_\circ(z_c) &= O(2d)^{-2}.
\end{align}
\end{lemma}

\begin{proof}
Since a lattice tree or lattice animal containing $0$ and $s$
must contain a path of length at least $1$ joining those vertices,
we have $G_{z_c}(s) \le S_{z_c}^{(1,1)} \le O(2d)^{-1/2}$, where the
last inequality follows from
Lemma~\ref{lem:DmGn}.  This proves \eqref{e:Gstart}.

The limit \eqref{e:rcbd} follows from the identity
$2dr_c  =2dz_c g_c -2dz_c G_{z_c}(s)$ of \eqref{e:rgG}, together
with \eqref{e:step0}--\eqref{e:zcbd0} and \eqref{e:Gstart}.

Finally, since the minimal length of a cycle containing the origin
in a lattice animal is $4$, it follows that $g_\circ(z_c) \le S^{(4,1)}_{z_c}$,
and then \eqref{e:gcirc0} is a consequence of Lemma~\ref{lem:DmGn}.
\end{proof}

\begin{lemma}
\label{lem:zcleading}
For lattice trees or lattice animals, $g_c = \ee +o(1)$ and
$z_c = (2d\ee)^{-1} + o(2d)^{-1}$.
\end{lemma}

\begin{proof}
According to \eqref{e:step0}, it suffices to prove that $g_c =\ee + o(1)$,
and this follows immediately from
Theorem~\ref{thm:gstart} and \eqref{e:rcbd}--\eqref{e:gcirc0}.
\end{proof}

\section{Second term}
\label{sec:term2}

In this section, we compute the $(2d)^{-2}$ term in the expansion
for $z_c$ in \eqref{e:zca3}, and the $(2d)^{-1}$ term
in the expansion for $g_c$ in \eqref{e:ga}.
We follow the strategy discussed in Section~\ref{sec:pfstructure}:
we first compute the $(2d)^{-1}$ terms in the expansions
for $G_{z_c}(s)$ and
for $\hat\Pi_{z_c}$ in \eqref{e:G0x_2}--\eqref{e:Pia-2terms},
then use this to compute the desired term for $g_c$, and finally
obtain the desired term for $z_c$.

A useful quantity is
\begin{equation}
\label{e:Qdef}
    Q(x)=
    \sum_{C_0\ni 0}   \sum_{C_x \ni x}
    z_c^{|C_0|+|C_x|} \1_{C_0 \cap C_x \neq \varnothing},
\end{equation}
where the sum is over clusters (both trees or both animals)
containing $0$ and $x$, respectively.
It is shown in Lemma~\ref{lem:Q} that for $s$ a neighbour
of the origin, and for both lattice trees and lattice animals,
the leading behaviour
\begin{equation}
\label{e:Qleading}
    Q(s)=2z_c^2g_c^3+o(2d)^{-1} =  \frac{2\ee}{2d} + o(2d)^{-1}
\end{equation}
arises
from the presence of the bond $\{0,s\}$ in one of the two
clusters $C_0$ or $C_1$.
The proof of Lemma~\ref{lem:Q} uses only Lemmas~\ref{lem:zcleading}
and \ref{lem:DmGn}.

\begin{lemma}
\label{lem:G1}
For lattice trees or lattice animals, and for a neighbour $s$ of the origin,
\begin{equation}
    G_{z_c}(s) = \frac{\ee}{2d} + o(2d)^{-1}.
\end{equation}
\end{lemma}

\begin{proof}
For a lattice tree or lattice animal containing $0$ and $s$,
either the bond $\{0,s\}$ is occupied or it is not.
In the latter case, there must be an occupied path connecting 0 and $s$
of length at least 3.
In the former case, we overcount with independent clusters
at $0$ and $s$.  This gives
\begin{equation}
    G_{z_c}(s) \le z_c g_c^2 + G_{z_c}^{(3)}(s)
    \le z_c g_c^2 + S_{z_c}^{(3,1)},
\end{equation}
where the last inequality comes from \eqref{eq:Smn}--\eqref{eq:Smn_def}.
By Lemmas~\ref{lem:zcleading} and \ref{lem:DmGn}, it follows that
\begin{equation}
    G_{z_c}(s) \le \frac{\ee}{2d} + o(2d)^{-1}.
\end{equation}
For a lower bound, we
consider only the case where the edge
$\{0,s\}$ is occupied and not part of a cycle (for lattice animals).
It follows from inclusion exclusion
that
\begin{equation}
    G_{z_c}(s) \ge z_c g_c^2
    - z_c Q(s),
\end{equation}
and it then follows from \eqref{e:Qleading} and Lemma~\ref{lem:zcleading} that
\begin{equation}
    G_{z_c}(s) \ge \frac{\ee}{2d} + o(2d)^{-1}.
\end{equation}
This completes the proof.
\end{proof}

\begin{lemma}
\label{lem:Pi1}
For lattice trees or lattice animals,
\begin{equation}
    \hat\Pi_{z_c} = -\frac{3\ee}{2d} + o(2d)^{-1}.
\end{equation}
\end{lemma}

\begin{proof}
It follows from \eqref{e:HSPibds} and \eqref{e:HSPibda}
that we need only consider the contributions due to $\hat\Pi_{z_c}^{(t,1)}$
for trees, and due to $\hat\Pi_{z_c}^{(a,0)}$ and $\hat\Pi_{z_c}^{(a,1)}$
for animals, since larger values of $N$ contribute $O(2d)^{-2}$.
Moreover, we can neglect $\hat\Pi_{z_c}^{(a,0)}$.
To see this, we recall the definition \eqref{eq:Pi0} and
apply the BK inequality of \cite[Lemma~2.1]{HS90b}
and Lemma~\ref{lem:DmGn} to see that
\begin{equation}
\label{e:Pia0}
    \hat\Pi_{z_c}^{(a,0)}
    \le \sum_{i+j =4} \sum_{x \in \Zd} G^{(i)}_{z_c}(x) G^{(j)}_{z_c}(x)
    =
    S_{z_c}^{(4,2)} \le O(2d)^{-2},
\end{equation}
where the restriction to $i+j=4$ arises because
only animals in which the origin is in a cycle of length at least 4 can occur.
Therefore, we can restrict attention to the case $N=1$.

By definition,
\begin{equation}\label{e:Pi1_decom}
    \hat\Pi_{z_c}^{(1)}
    =
    \Pi_{z_c}^{(1)}(0)
    +
    \sum_{s : \|s\|_1=1} \Pi_{z_c}^{(1)}(s)
    +
    \sum_{x : \|x\|_1 \ge 2} \Pi_{z_c}^{(1)}(x).
\end{equation}
A nonzero contribution to $\Pi^{(1)}_{z_c}(x)$
requires the existence of three bond-disjoint paths as indicated
in Figure~\ref{fig:Pi1} (with $y=0$ or $y=x$ allowed), to ensure that
$\U_{0|\omega|}=-1$ in \eqref{Pi_uno} or $\U_{0|B|}=-1$ in \eqref{LA_Pi_uno}.
This implies that
\begin{equation}
\label{e:Pi1crude}
    \hat\Pi^{(1)}_{z_c}
    \le \sum_{x,y \in \Zd}
    G_{z_c}(x)G_{z_c}(y)G_{z_c}(y-x) = S_{z_c}^{(0,3)}(0)
    \le S_{z_c}^{(0,3)};
\end{equation}
a detailed derivation of this estimate can  be found, e.g.,
in \cite[Theorem~8.2]{Slad06}.
The crude bound \eqref{e:Pi1crude}
can be greatly improved by replacing two-point
functions by factors $G^{(i)}_{z_c}$ when there must be
at least $i$ steps taken.  In this way, for contributions to
$\hat\Pi^{(1)}_{z_c}$ in which there must exist paths from $0$ to $x$,
from $0$ to $y$, and from $x$ to $y$, of total length at least $m$,
we can improve the upper bound $S_{z_c}^{(0,3)}$ to $S_{z_c}^{(m,3)}
\le O(2d)^{-m/2}$.
In particular, this implies that the
last sum on the right-hand side of \eqref{e:Pi1_decom}
is bounded by $S^{(4,3)}_{z_c} \le O(2d)^{-2}$ and thus is an error term.

\begin{figure}[h]
\centering
{
\includegraphics[scale=.09]{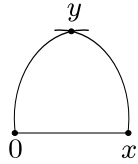}
}
\caption{Intersection required for a nonzero contribution to $\Pi^{(1)}_z(x)$.
\label{fig:Pi1}
}
\end{figure}

The leading behaviour arises from the other two terms.
We consider both trees and animals simultaneously.
Consider first the lower bound.
For $\Pi_{z_c}^{(1)}(0)$, we count only configurations with
backbone $(0,s,0)$ where $\|s\|_1=1$.  By using inclusion-exclusion
to account for the avoidance between the rib at $s$ and the two ribs
at $0$, we obtain
\begin{equation}
    \Pi_{z_c}^{(1)}(0)
    \ge
    2dz_c^2 (g_c^3 - 2g_cQ(s))
    =
    \frac{\ee}{2d} + o(2d)^{-1}
    ,
\end{equation}
by Lemma~\ref{lem:zcleading} and \eqref{e:Qleading}.  Similarly, by
considering the symmetric cases where either the rib at $0$
contains $\{0,s\}$ or the rib at $s$ contains $\{0,s\}$, we obtain
\begin{equation}
    \Pi_{z_c}^{(1)}(s)
    \ge
    2z_c^2 (g_c^3 - g_cQ(s)),
\end{equation}
and hence
\begin{equation}
    \sum_{s : \|s\|_1=1} \Pi_{z_c}^{(1)}(s)
     \ge
    \frac{2\ee}{2d} + o(2d)^{-1}
    .
\end{equation}
Altogether, this gives
\begin{equation}
    \hat\Pi_{z_c}^{(1)}
    \ge
    \frac{3\ee}{2d} + o(2d)^{-1}.
\end{equation}

For the upper bound,
excepting the configurations which contributed the leading behaviour to the
lower bound, the remaining configurations that contribute to
$\hat\Pi_{z_c}^{(1)}$ all contain three paths of total length at least 4,
and hence are bounded above by $S_{z_c}^{(4,3)}\leq O(2d)^{-2}$.  This completes the proof.
\end{proof}

\begin{lemma}
\label{lem:g2}
For lattice trees or lattice animals,
\begin{align}
\label{e:gct2}
    g_c &= \ee \left[1 +\frac{\frac 32}{2d} \right] + o(2d)^{-1}
    ,
    \\
\label{e:zct2}
    z_c    &=\ee^{-1}\left[ \frac{1}{2d}+\frac{\frac{3}{2}}{(2d)^2}
     \right] +   o(2d)^{-2}
    .
\end{align}
\end{lemma}

\begin{proof}
We begin by noting that $g_\circ(z_c) =
O(2d)^{-1}$, by \eqref{e:gcirc0}.
Next, we combine the identity $2dr_c =
1 - 2d\z{3}\FPI{3} - 2d\z{3}G_{z_c}(s)$ of \eqref{e:rPiG} with
Lemmas~\ref{lem:zcleading} and \ref{lem:G1}--\ref{lem:Pi1}
to obtain
\begin{equation}
\label{e:1-M}
    2dr_c =
   1 + \frac{3}{2d}
   - \frac{1}{2d} + o(2d)^{-1}
   = 1 + \frac{2}{2d} + o(2d)^{-1}.
\end{equation}
Then \eqref{e:gct2} follows immediately after substitution of
\eqref{e:1-M} into the right-hand side of
the identity for $g_c$ in Theorem~\ref{thm:gstart}.
Finally, \eqref{e:zct2}
follows from substitution of \eqref{e:gct2} and the formula
for $\hat\Pi_{z_c}$ of Lemma~\ref{lem:Pi1} into \eqref{e:zc-bis}.
\end{proof}

\section{Third term}
\label{sec:term3}

We now complete the proof of Theorem~\ref{thm:1}.  To do this,
we first extend the estimates on $G_{z_c}(s)$ and
$\hat\Pi_{z_c}$ obtained
in Lemmas~\ref{lem:G1}--\ref{lem:Pi1}.
With these extensions, we then extend the estimate
on $g_c$ of Lemma~\ref{lem:g2},
and finally combine these results with \eqref{e:zc-bis} to extend the
estimate on $z_c$ and thereby complete the
proof of Theorem~\ref{thm:1}.
To begin, we insert the formulas of Lemma~\ref{lem:g2} into the formula
for $Q(s)$ of Lemma~\ref{lem:Q}, to obtain
\begin{equation}
\label{e:Q2}
    Q(s)
    =  2 z_cg_c^3  -  \frac{ \ee^2 }{(2d)^2}  +  o(2d)^{-2}
    =  \ee^2
    \left[
    \frac{2}{2d}  +  \frac{ 11 }{(2d)^2}
    \right]
    +  o(2d)^{-2}.
\end{equation}

The estimate we need for $G_{z_c}(s)$ was stated earlier as
Theorem~\ref{thm:G_2}, which for convenience we restate
as follows.

\begin{theorem}
\label{thm:G_2-bis}
For lattice trees or lattice animals,
and for a neighbour $s$ of the origin,
\begin{equation}
\label{e:G0x_2-bis}
    \GG{3}(s)
    =
    \ee
    \left[
    \frac{1}{2d} + \frac{\frac72}{(2d)^2}
    \right]
    + o(2d)^{-2}.
\end{equation}
\end{theorem}

\begin{proof}
It follows from Lemma~\ref{lem:DmGn}
that $G_{z_c}^{(5)}(s) \le O(2d)^{-5/2}$, so we need only consider
clusters in which a path of length $1$ or $3$ joins the points $0$
and $s$.

For the lower bound, we consider clusters that either contain
the bond $\{0,s\}$ with this bond not in a cycle,
or that do not contain $\{0,s\}$ but contain
one of the $2d-2$ paths of length $3$ from $0$ to $s$ with this path not
part of a cycle.
The first contribution is equal to
\begin{equation}
    z_c (g_c^2 -Q(s))
    =
    \ee \left[ \frac{1}{2d} + \frac{\frac52}{(2d)^2}
    \right] + o(2d)^{-2},
\end{equation}
by \eqref{e:Q2} and Lemma~\ref{lem:g2}.
With $s'$ a neighbour of the origin that is not equal to
$\pm s$, the second contribution is bounded below by
\begin{align}
    &(2d-2) z_c^3
    \sum_{R_0\ni0,R_1\ni s' \atop R_2\ni s+s',R_3\ni s}
    z_c^{\abs{R_0}+\abs{R_1}+\abs{R_2}+\abs{R_3}}
    \prod_{0\leq i<j\le 3}\ob{1+\U_{ij}}
    \nonumber \\ \nonumber
    &\quad
    \geq (2d-2) z_c^3
    \sum_{R_0\ni0,R_1\ni s' \atop R_2\ni s+s',R_3\ni s}
    z_c^{\abs{R_0}+\abs{R_1}+\abs{R_2}+\abs{R_3}}
    \ob{1+\sum_{0\leq i<j\leq3}\U_{ij}}\\
    &\quad
    =(2d-2) z_c^3\left( g_c^4 - 4g_c^2 Q(s)
    - 2g_c^2 Q(s+s') \right).
\end{align}
Now we apply Lemma~\ref{lem:zcleading}, and the fact that
$Q(x)=o(1)$ by Lemma~\ref{lem:Q},
to see that this last
expression is equal to
\begin{equation}
    (2d-2) z_c^3\left( g_c^4 - 4g_c^2 Q(s)
    - 2g_c^2 Q(s+s') \right)
    =
    \frac{\ee}{(2d)^2} +o(2d)^{-2}.
\end{equation}
Combining the above results gives the lower bound
\begin{equation}
    \GG{3}(s)
    \ge
    \ee
    \left[
    \frac{1}{2d} + \frac{\frac72}{(2d)^2}
    \right]
    + o(2d)^{-2}.
\end{equation}

For an upper bound, we again need only
consider the cases where there is a path of length
1 or 3 connecting $0$ and $s$.  Suppose first that there is a
path of length 1.  If the bond $\{0,s\}$ is not in a cycle, then the
above argument again gives a contribution
\begin{equation}
    z_c (g_c^2 -Q(s))
    =
    \ee \left[ \frac{1}{2d} + \frac{\frac52}{(2d)^2}
    \right] + o(2d)^{-2}.
\end{equation}
On the other hand, if $\{0,s\}$ is part of a cycle, then we need
only consider the case where this bond is part of
a cycle of length 4, because otherwise
there is a path from $0$ to $s$ of length at least $5$.
The contribution from animals containing $\{0,s\}$ within a cycle of length 4
is at most $(2d-2)z_c^4 g_c^4 = O(2d)^{-3}$, so this is an error
term.  Thus the upper bound for the case of direct connection agrees
with the lower bound.   In addition, the contribution
when there is a path of length 3 is at most
\begin{equation}
    (2d-2)z_c^3 g_c^4
    =
    \frac{\ee}{(2d)^2} + o(2d)^{-2},
\end{equation}
so here too the upper and lower  bounds match, and the proof is complete.
\end{proof}

Next, we present three lemmas which extract the terms
in $\hat\Pi_{z_c}^{(N)}$ up to $o(2d)^{-2}$,
for $N=0,1,2$.  The case $N=0$ occurs only for lattice animals,
and we begin with this case.

\begin{lemma}
\label{lem:Pi0}
For lattice animals,
\begin{align}
\label{e:Pia0-2}
    \hat\Pi^{(a,0)}_{z_c}
    &=  \frac{ \frac32 }{(2d)^2} + o (2d)^{-2},
    \\
\label{e:gcirc-1}
    g_\circ (z_c^{(a)})
    & =
    \frac{\frac 12}{(2d)^2} + o(2d)^{-2}.
\end{align}
\end{lemma}

\begin{proof}
According to its definition in \eqref{eq:Pi0},
\begin{equation}
    \hat\Pi^{(a,0)}_{z_c}
    =
    \sum_{x \neq 0} \sum_{R \in {\cal D}_{0,x}} z_c^{|R|}.
\end{equation}
The main contribution to the right-hand side arises when
$R$ is a unit square containing $0$, with $x$ a nonzero
vertex on the square.
Therefore, since there are $\frac 12 (2d)(2d-2)$
such squares and $3$ nonzero
vertices in each one,
\begin{equation}
    \hat\Pi^{(a,0)}_{z_c}
    \le
    3\frac{1}{2} (2d)(2d-2) z_c^4 g_c^4 + S^{(6,2)}
    =
    \frac{ \frac32 }{(2d)^2} + o (2d)^{-2} ,
\end{equation}
where we used Lemmas~\ref{lem:zcleading} and \ref{lem:DmGn} in the
last equality.
For a lower bound, we count only the contributions with $0,x$
in a cycle of length $4$,
and use inclusion-exclusion for the branches emanating from the unit
square,
to obtain
\begin{equation}
    \hat\Pi^{(a,0)}_{z_c}
    \ge
    3\frac{1}{2} (2d)(2d-2) z_c^4
    \left[ g_c^4 -4 g_c^2 Q(s_1) - 2 g_c^2 Q(s+s')) \right]
    =
    \frac{ \frac32 }{(2d)^2} + o (2d)^{-2} ,
\end{equation}
where we have used Lemma~\ref{lem:zcleading} together with
the fact that $Q(x)=o(1)$ by Lemma~\ref{lem:Q}.
This proves \eqref{e:Pia0-2}.

A similar argument gives \eqref{e:gcirc-1}, with the factor $3$ missing
due to the fact that there is no sum over $x$ in $g_\circ$.
\end{proof}

\begin{lemma}
\label{lem:Pi1-2}
For lattice trees or lattice animals,
\begin{align}
\label{e:Pi1-2}
    \hat\Pi^{(1)}_{z_c}
    & =  \ee \left[
    \frac {3}{2d} + \frac{\frac{49}{2}}{(2d)^2}\right] + o(2d)^{-2}
    .
\end{align}
\end{lemma}

\begin{proof}
We give the proof only for the case of
lattice trees.  With minor changes, the arguments
presented here also lead to a proof for lattice animals.

By definition,
\begin{equation}
    \hat\Pi^{(1)}_{z_c} = \sum_{x \in \Zd} \Pi^{(1)}_{z_c}(x).
\end{equation}
Contributions from $x \neq 0, s, s+s'$, where $s,s'$ are orthogonal
neighbours of the origin, are bounded above by $S^{(6,3)}=O(2d)^{-3}$
and need not be considered further.  By symmetry, we therefore have
\begin{equation}
    \hat\Pi^{(1)}_{z_c} =  \Pi^{(1)}_{z_c}(0) + 2d \Pi^{(1)}_{z_c}(s)
    + \frac{2d(2d-2)}{2}
      \Pi^{(1)}_{z_c}(s+s') + O(2d)^{-3}.
\end{equation}

Consider $\Pi^{(1)}_{z_c}(s+s')$.  The shortest backbones have
length $2$ and there are $2$ of these.  The shortest allowed
rib intersections complete the unit square and there are 3 of
these corresponding to the 3 possible nonzero intersection
points for the ribs at 0 and $s+s'$.  Thus we obtain
\begin{equation}
    \Pi^{(1)}_{z_c}(s+s')
    \le 3 \cdot 2 z_c^4 g_c^5 + S^{(6,3)}
    +  O(2d)^{-3}
    \le \frac{6\ee}{(2d)^4} + O(2d)^{-3}.
\end{equation}
Arguments like those we have been using previously
can be used to verify that the first term on the right-hand
side is also the leading behaviour of a lower bound, and hence
\begin{equation}
    \Pi^{(1)}_{z_c}(s+s')
    = \frac{6\ee}{(2d)^4} + o(2d)^{-2}.
\end{equation}
This shows that
\begin{equation}
    \hat\Pi^{(1)}_{z_c} =  \Pi^{(1)}_{z_c}(0) + 2d \Pi^{(1)}_{z_c}(s)
    + \frac{3\ee}{(2d)^2} + o(2d)^{-2}.
\end{equation}

Consider $\Pi^{(1)}_{z_c}(s)$.
We need only consider the contributions
due to rib intersections which together with the backbone form
a double bond or
a unit square, because the remaining contributions are
bounded by $S^{(6,3)}=O(2d)^{-3}$.
These backbones have length  $1$ or $3$, respectively.
Thus we obtain
(the first term is due to the
length-1 backbone
and the second to the length-3 backbone)
\begin{align}
    \Pi^{(1)}_{z_c}(s)
    & \le
    z_c Q(s) +
     (2d-2) z_c^3 g_c^2 Q(s) + O(2d)^{-3}
    \nonumber \\
    & = \ee
    \left[ \frac{2}{(2d)^2} + \frac{16}{(2d)^3} \right] + o(2d)^{-2},
\end{align}
by Lemma~\ref{lem:zcleading} and \eqref{e:Q2}.
It is routine to prove a matching lower bound, yielding
\begin{align}
    2d \Pi^{(1)}_{z_c}(s)
    & = \ee \left[ \frac{2}{2d} + \frac{16}{(2d)^2} \right] + o(2d)^{-2},
\end{align}
and hence
\begin{equation}
    \hat\Pi^{(1)}_{z_c} =  \Pi^{(1)}_{z_c}(0) +
    \ee \left[ \frac{2}{2d} + \frac{19}{(2d)^2} \right] + o(2d)^{-2}.
\end{equation}

Finally, we consider the contributions to $\Pi^{(1)}_{z_c}(0)$ due
to backbones of length $2$ and $4$, which we denote as
$\Pi^{(1,2)}_{z_c}(0)$ and $\Pi^{(1,4)}_{z_c}(0)$ respectively.
First,
\begin{align}
    \Pi^{(1,4)}_{z_c}(0)
    & \le
    2d(2d-2) z_c^4 g_c^5 + S^{(6,1)}
    =  \frac{\ee}{(2d)^2}
    +O(2d)^{-3},
\end{align}
and a routine matching lower bound gives
\begin{align}
    \Pi^{(1,4)}_{z_c}(0)
    &
    =  \frac{\ee}{(2d)^2}
    +O(2d)^{-3}.
\end{align}
Next,
\begin{align}
\label{eq:Pi1_2_A}
    \Pi^{(1,2)}_{z_c}(0)
    &= 2d \z{3} ^2\sum_{\substack{R_0\ni0,R_1\ni s \\ R_2\ni0}}
    \z{3}^{\abs{R_0}+\abs{R_1}+\abs{R_2}}
    \ob{1+\U_{01}+\U_{12}+\U_{01}\U_{12}}
    \nonumber
    \\
    &=
    2d \z{3}^2\Big[ g_c^3 - 2 g_c Q(s)
    + \sum_{\substack{R_0\ni0,R_1\ni s \\ R_2\ni0}}
    \z{3}^{\abs{R_0}+\abs{R_1}+\abs{R_2}} \U_{01}\U_{12} \Big]
    \nonumber
    \\
    & =
    \ee \left[ \frac{1}{2d} + \frac{\frac72}{(2d)^2} \right]
    + o(2d)^{-2}
    +
     2d \z{3}^2
     \left[\frac{\ee^3}{2d} +  o(2d)^{-1}\right]
     ,
\end{align}
where we used Lemma~\ref{lem:g2} and
Lemmas~\ref{lem:Q}--\ref{lem:QPi1} in the last equality.
With Lemma~\ref{lem:zcleading}, this gives
\ecus{Pi1_2}{
        \Pi^{(1,2)}_{z_c}(0)
        & = \frac{\ee}{2d} + \frac{\frac92\ee}{(2d)^2}
        + o(2d)^{-2}.
}
Thus we obtain
\begin{align}
    \Pi^{(1)}_{z_c}(0)
    &
    = \ee \left[ \frac{1}{2d} + \frac{\frac{11}{2}}{(2d)^2} \right]
    +o(2d)^{-2}.
\end{align}

Altogether, we have
\begin{equation}
    \hat\Pi^{(1)}_{z_c} =
    \ee \left[ \frac{3}{2d}
    + \frac{\frac{49}{2}}{(2d)^2} \right] + o(2d)^{-2},
\end{equation}
which proves \eqref{e:Pi1-2}.
\end{proof}

\begin{lemma}
\label{lem:Pi2}
For lattice trees or lattice animals,
\begin{align}
\label{e:Pi2-2}
    \hat\Pi^{(2)}_{z_c}
    & =  \frac{11\ee}{(2d)^2}  + o(2d)^{-2}.
\end{align}
\end{lemma}

\begin{proof}
We defer the proof to Lemma~\ref{lem:Pi2-bis}.
\end{proof}

For convenience, we now restate Theorem~\ref{thm:Pi-2terms},
supplemented with an asymptotic formula for $g_\circ(z_c^{(a)})$.
Note that the factor $\ee$ is not present for $g_\circ(z_c^{(a)})$.
It is in Theorem~\ref{thm:Pi-2terms-bis}
that we first see a difference between lattice trees
and lattice animals.

\begin{theorem}
\label{thm:Pi-2terms-bis}
For lattice trees or lattice animals,
\begin{align}
\label{e:Pia-2terms-bis}
    \hat\Pi_{z_c}
    &=
    \ee\left[
    -\frac{3}{2d} - \frac{ \frac{27}2 - \1_{\rm a}\frac32 \ee^{-1}}{(2d)^2}
    \right]
    +  o(2d)^{-2},
    \\
\label{e:gcirc}
    g_\circ (z_c^{(a)})
    & =
    \frac{\1_{\rm a} \frac 12}{(2d)^2} + o(2d)^{-2}.
\end{align}
\end{theorem}

\begin{proof}
This follows immediately from
Lemmas~\ref{lem:Pi0}--\ref{lem:Pi2}, together with the bounds
on $\hat\Pi_{z_c}^{(N)}$ for
$N>2$ given by \eqref{e:HSPibds} and \eqref{e:HSPibda}.
\end{proof}

The next theorem restates Theorem~\ref{thm:1}, and completes its proof
(apart from the technical lemmas of Section~\ref{sec:cie}).

\begin{theorem}
\label{thm:g}
For lattice trees or lattice animals,
\begin{align}
\label{e:ga-bis}
    g_c  &= \ee
    \left[ 1 +\frac{\frac{3}{2}}{2d}
    +\frac{ \frac{263}{24} - \1_{\rm a} \ee^{-1} }{(2d)^2}\right]
    +o(2d)^{-2},
    \\
\label{e:zca3-bis}
    z_c   &=\ee^{-1}\left[ \frac{1}{2d}+\frac{\frac{3}{2}}{(2d)^2}
    +\frac{\frac{115}{24} -
    \1_{\rm a} \frac1{2}\ee^{-1} }{(2d)^3} \right] + o(2d)^{-3}
    .
\end{align}
\end{theorem}

\begin{proof}
By \eqref{e:gstart} and \eqref{e:gcirc},
\begin{align}
\label{e:gend}
    g_c & =  \ee^{2dr_c}
    \left[
    1 - \frac 12 (2d)r_c^2 + \frac{1}{8}(2d)^2 r_c^4  - \frac{\frac 76}{(2d)^2}  \right]
    + \frac{\1_{\rm a} \frac 12}{(2d)^2}
    +o(2d)^{-2}  .
\end{align}
The identity \eqref{e:rPiG}, together with the results for
$z_c$, $\FPI{3}$ and $G_{z_c}(s)$ in
Lemma~\ref{lem:g2} and
Theorems~\ref{thm:G_2-bis} and \ref{thm:Pi-2terms-bis}, imply that
\begin{equation}
\label{e:A3M1}
    2dr_c
    =
    1- 2d\z{3}\FPI{3} -2dz_c G_{z_c}(s)
    =
    1 +\frac2{2d} + \frac{ 13 - \1_{\rm a} \frac{3}{2}\ee^{-1} }{(2d)^2}
    + o(2d)^{-2}.
\end{equation}
Substitution of \eqref{e:A3M1} into \eqref{e:gend} gives
\eqref{e:ga-bis}.
Finally, \eqref{e:zca3-bis}
follows immediately by substituting \eqref{e:Pia-2terms-bis}
and \eqref{e:ga-bis} into \eqref{e:zc-bis}.
\end{proof}

\section{Cluster intersection estimates}
\label{sec:cie}

The analysis in Sections~\ref{sec:onept}, \ref{sec:term1}, \ref{sec:term2},
and \ref{sec:term3}
relies on the estimates in this section, which in turn rely on
Lemma~\ref{lem:DmGn}.
Section~\ref{sec:Zestimates} provides the estimates needed for
the proof of
Theorem~\ref{thm:gstart}, and assumes only the starting bounds
\eqref{e:step0}.
Section~\ref{sec:Qestimates}
provides estimates needed in Sections~\ref{sec:term2}--\ref{sec:term3},
and relies on knowledge of the leading behaviour $g_c \sim \ee$ and
$z_c \sim (2d\ee)^{-1}$.

\subsection{Estimates for one-point function}
\label{sec:Zestimates}

Throughout this section, we assume only the starting bounds \eqref{e:step0}
and do not make use of higher order asymptotics.
In particular, we will make use of \eqref{e:rcbd}, which states that
$2dr_c = 1+o(1)$.
We prove two lemmas which provide
estimates needed in the proof
of Theorem~\ref{thm:gstart}.
The first gives estimates for $Z^{(i)}$ ($i=1,2,3$) defined
in \eqref{e:Zdef}--\eqref{e:Zstar2}, as well as
for $Z',Z''$ defined by
\begin{align}
\label{e:Zprime}
    Z^{'}(z)
    &=
    \sum_{s_1,s_2,s_3,s_4 \in \setE}
    \sum_{S_1\ni s_1}\sum_{S_2 \ni s_2}\sum_{S_3\ni s_3}\sum_{S_4\ni s_4}
    \z{4}^{\abs{S_1}+ \abs{S_2} +\abs{S_3} + \abs{S_4}} (-\V_{12}\V_{13}\V_{14}),
    \\
\label{e:Zprime2}
    Z^{''}(z)
    &=
    \sum_{s_1,s_2,s_3,s_4 \in \setE}
    \sum_{S_1\ni s_1}\sum_{S_2 \ni s_2}\sum_{S_3\ni s_3}\sum_{S_4\ni s_4}
    \z{4}^{\abs{S_1}+ \abs{S_2} +\abs{S_3} + \abs{S_4}} (-\V_{12}\V_{13}\V_{24})
    .
\end{align}
We use the subscript $c$ to denote quantities evaluated at $z_c$.

\begin{lemma}
\label{lem:ZZZ}
For lattice trees or lattice animals,
\begin{align}
\label{e:Zbd}
    Z^{(1)}_c
    &=
    2dr_c^2 + \frac{3}{(2d)^2} + o(2d)^{-2}
    ,
    \\
\label{e:Zstarbd}
    Z^{(2)}_c
    &=
    \frac{ 1 }{(2d)^2}
    + o(2d)^{-2}
    ,
    \\
\label{e:Zstar2bd}
    Z^{(3)}_c
    & =
    \frac{ 1}{(2d)^2}
    + o(2d)^{-2}
    ,
    \\
\label{e:Zprimes}
    Z^{'}_c
    & =
    Z^{''}_c
    = O(2d)^{-3}.
\end{align}
\end{lemma}

\begin{proof}  We consider the four equations in turn.

\smallskip \noindent \emph{Proof of \eqref{e:Zbd}.}
According to its definition in \eqref{e:Zdef},
\begin{align}
\label{e:Zdef-bis}
    Z^{(1)}_c &=
    \sum_{s_1,s_2\in \setE}
    \sum_{S_1\ni s_1} \sum_{S_2\ni s_2}
    z_c^{\abs{S_1} + \abs{S_2}}  (-{\cal V}_{12}).
\end{align}
We distinguish the two possibilities
$\abs{ \set{s_1,s_2} }=1,2$ for the vertices $s_1$, $s_2$, i.e.,
we distinguish whether the two vertices are equal or not.
If $s_1=s_2$ then automatically $-\V_{12}= 1$ because both
clusters contain $s_1$, and this contribution
gives exactly $2d r_c^2$.

For $s_1 \neq s_2$, we consider separately
the cases where $s_1$ and $s_2$ are parallel and perpendicular.
This contributes
\begin{align}\label{e:Zstar_dif}
    &
    2d \sum_{S_1\ni s_1\atop S_2\ni -s_1}
    \z{3}^{\abs{S_1}+\abs{S_2}} (-\V_{12})
    \;\;
    +
    \;\;
    2d(2d-2)\sum_{S_1\ni s_1\atop S_2\ni s_2}
    \z{3}^{\abs{S_1}+\abs{S_2}} (-\V_{12})
    .
\end{align}
For the first term, at least six steps are required for an intersection
of $S_1$ and $S_2$, so this term is bounded above by
$S_{z_c}^{(6,2)} =O(2d)^{-3}$, by Lemma~\ref{lem:DmGn}.  The leading behaviour
of the second term  is $3(2d)(2d-2)z_c^4 g(z_c)^4$
(due to a square containing $0,s_1$ and $s_2$, and where the factor $3$
takes into account the three nonzero vertices of the square at
which $S_1,S_2$ might intersect).  The remaining contributions are bounded
above by $S_{z_c}^{(6,2)} =O(2d)^{-3}$.  It is not difficult to prove a
corresponding lower bound,
to conclude that \eqref{e:Zstar_dif} equals
\begin{align}\label{e:Zstar_dif2}
    3 (2d)^2 z_c^4 g_c^4 + o(2d)^{-2}
    &
    =
    \frac{3}{(2d)^2} + o(2d)^{-2}
    ,
\end{align}
where we used \eqref{e:step0} for the last equality.
When combined with the contribution from
$s_1=s_2$, this completes the proof of \eqref{e:Zbd}.

\smallskip \noindent \emph{Proof of \eqref{e:Zstarbd}.}
By definition
\begin{align}
    Z^{(2)}_c
    &=
    \sum_{s_1,s_2,s_3 \in \setE}
    \sum_{S_1\ni s_1}\sum_{S_2 \ni s_2}\sum_{S_3\ni s_3}
    \z{3}^{\abs{S_1}+ \abs{S_2} +\abs{S_3}} \V_{12}\V_{13}.
\end{align}
We distinguish the three possibilities
$\abs{ \set{s_1,s_2,s_3} }=1,2,3$ for the vertices $s_1$, $s_2$
and $s_3$.

If $\abs{ \set{s_1,s_2,s_3} }=1$, then automatically
$\V_{12}\V_{13}=1$.
In this case, using \eqref{e:rcbd}, we find that the contribution
to $Z^{(2)}_c$
becomes simply
\ecu{e:C13-zzz}{
    2dr_c^3
    =
    \frac{ 1 }{(2d)^2}
    + o(2d)^{-2}.
}

If
$\abs{ \set{s_1,s_2,s_3} }= 2$, then we consider the case
$s_1 \neq s_2 = s_3$ (the other cases can be handled with
similar arguments).  In this case, we
use $|\V_{13}|\le 1$ and perform the sum over $S_3$ to obtain
a factor $r_c$.  The remaining sum is the case $s_1\neq s_2$ studied
in the bound on $Z^{(1)}_c$ and shown above in \eqref{e:Zstar_dif2}
to be $O(2d)^{-2}$.
Thus this contribution is an error term, since $r_c=O(2d)^{-1}$ by \eqref{e:rcbd}.

If $\abs{ \set{s_1,s_2,s_3} }= 3$, then all three
vertices are different.
At least seven bonds are required to obtain $\V_{12}\V_{13}=1$ in
this case.  As depicted in Figure~\ref{fig:Z2bd}, this contribution
is bounded above by $\sum_{i+j =7}S_{z_c}^{(i,2)}S_{z_c}^{(j,3)}$ and hence is
$O(2d)^{-7/2}$, another error term.
This completes the proof of \eqref{e:Zstarbd}.

\begin{figure}[!h]
 \centering
 \includegraphics[scale=2.8]{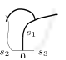}
 \quad\quad
\includegraphics[scale=2.8]{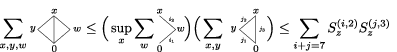}
 \caption{Intersections required for the case $\abs{ \set{s_1,s_2,s_3} }= 3$
 of $Z^{(2)}$, with corresponding bound.
 The paths have lengths at least $i_k$ and $j_l$,
 with $i_1+i_2=i$, $j_1+j_2+j_3=j$, and $i+j=7$.}
 \label{fig:Z2bd}
\end{figure}

\smallskip \noindent \emph{Proof of \eqref{e:Zstar2bd}.}
By definition,
\begin{align}
    Z^{(3)}_c
    & =
    \sum_{s_1,s_2,s_3 \in \setE}
    \sum_{S_1\ni s_1}\sum_{S_2 \ni s_2}\sum_{S_3\ni s_3}
    \z{3}^{\abs{S_1}+ \abs{S_2} +\abs{S_3}} (-\V_{12}\V_{13}\V_{23})
    .
\end{align}
If $\abs{ \set{s_1,s_2,s_3} }=1$, then automatically
$-\V_{12}\V_{13}\V_{23}=1$ and hence
\begin{equation}
    Z^{(3)}_c
    \ge
    2dr_c^3 = \frac{1}{(2d)^2}
    + o(2d)^{-2}.
\end{equation}
On the other hand the inequality $-\V_{23} \le 1$, together
with \eqref{e:Zstarbd}, shows that
\begin{equation}
    Z^{(3)}_c
    \le
    Z^{(2)}_c = \frac{1}{(2d)^2}
    + o(2d)^{-2}.
\end{equation}
This completes the proof of \eqref{e:Zstar2bd}.

\smallskip \noindent \emph{Proof of \eqref{e:Zprimes}.}
We prove that $Z^{'}_c$ and $Z^{''}_c$ are $O(2d)^{-3}$.
For each,
we distinguish the four possibilities $\abs{ \set{s_1,s_2,s_3,s_4} }=1,2,3,4$.

If $\abs{ \set{s_1,s_2,s_3,s_4} }=1$,
the products $-\V_{12}\V_{13}\V_{14}=1$ and $-\V_{12}\V_{13}\V_{24}$ are
equal to $1$,
and the sums in \eqref{e:Zprime} and \eqref{e:Zprime2} reduce to $2dr_c^4$,
which is $O(2d)^{-3}$ by \eqref{e:rcbd}.

If $\abs{ \set{s_1,s_2,s_3,s_4} }=2$,
we decompose the products $-\V_{12}\V_{13}\V_{14}$
and $-\V_{12}\V_{13}\V_{24}$
into a factor that involves the two different vertices,
and the remaining two factors.
We bound these last two factors by 1,
so their corresponding sums are bounded by $r_c^2=O(2d)^{-2}$.
The remaining sums are equal to \eqref{e:Zstar_dif},
which by \eqref{e:Zstar_dif2} is of order $O(2d)^{-2}$.

If $\abs{ \set{s_1,s_2,s_3,s_4} }=3$,
we decompose $-\V_{12}\V_{13}\V_{14}$ and $-\V_{12}\V_{13}\V_{24}$
into two factors involving the three
distinct vertices,
and one remaining factor.
We bound the latter factor by 1,
and the corresponding sum becomes $r_c=O(2d)^{-1}$.
The remaining sums are bounded by $Z^{(2)}_c$ for the case $-\V_{12}\V_{13}\V_{14}$,
and by $Z^{(2)}_c$ or $(Z^{(1)}_c)^2$ for the case $-\V_{12}\V_{13}\V_{24}$.
The overall contribution in both cases is
therefore $O(2d)^{-3}$, by \eqref{e:Zbd} and \eqref{e:Zstarbd},

If $\abs{ \set{s_1,s_2,s_3,s_4} }=4$,
an example of the required intersections for $Z'$ is
depicted in Figure~\ref{fig:bdZprime}.  By taking into account
all possibilities for $Z'$ and $Z''$, we draw the crude
conclusion that at least six bonds are needed to achieve the
required intersections, and this leads to an upper bound of the
form $\sum_{n_1+n_2+n_3=6} O(S_{z_c}^{(n_1,M)}S_{z_c}^{(n_2,M)}S_{z_c}^{(n_3,M)})$,
for a fixed value of $M$, and hence is $O(2d)^{-3}$.

This completes the proof of \eqref{e:Zprimes} and of the lemma.
\end{proof}

\begin{figure}[!h]
 \centering
 \includegraphics[scale=2.8]{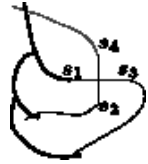}
 \caption{Example of intersections
 for the case $\abs{ \set{s_1,s_2,s_3,s_4} }= 4$ of $Z'$.}
 \label{fig:bdZprime}
\end{figure}

\begin{lemma}
\label{lem:Gamma34}
For lattice trees or lattice animals,
\begin{align}
\label{e:Gamma3bd}
    \Gamma^{(3,n)}_c & =  O(2d)^{-3} \quad \quad (n=4,5,6),
    \\
\label{e:Gamma4bd}
    \tilde\Gamma^{(4)}_c & = O(2d)^{-3}.
\end{align}
\end{lemma}

\begin{proof}
We consider the two equations in turn.

\smallskip \noindent \emph{Proof of \eqref{e:Gamma3bd}.}
First we consider $\Gamma^{(3,4)}$, and will show that
\begin{align}
\label{e:Gamma34}
    \Gamma^{(3,4)}(z)
    &= \Gamma^{(0)}(z )
    \left( \frac{4}{4!}Z^{'}(z)
    +
    \frac{12}{4!} Z^{''}(z) \right)    .
\end{align}
This is sufficient, by \eqref{e:Zprimes} together with the fact
that $\Gamma^{(0)}_c=\ee^{ 2dr_c } = O(1)$ by \eqref{e:Ar}
and \eqref{e:rcbd}.
To prove \eqref{e:Gamma34}, we are considering the case where
the set of labels $\set{i,j,k,l,p,q}$ in
\eqref{e:Def_Jcal3} has cardinality $4$,
and we may assume the labels are $1,2,3,4$.
We find 16 possible arrangements for
the labels, which
can be reduced to the two cases:
\\ (i)  Three labels are equal and the other three are different
    from the first ones and among them,
    e.g., $i=k=p=1$, $j=2$, $l=3$ and $q=4$.
    There are 4 arrangements of this type.
\\
(ii)  There are two pairs of equal labels and a pair of distinct labels,
    e.g., $i=k=1$, $j=p=2$, $l=3$ and $q=4$.
    There are 12 arrangements of this type.
\\
Interchanging the sums in which arise from
substitution of \eqref{e:Def_Jcal3} into \eqref{e:B}
(with $i=3$) and using symmetry,
as in the proof of Lemma~\ref{lem:Gams},
gives \eqref{e:Gamma34}.

For $\Gamma^{(3,5)}_c$, one of the factors
$\V_{ij}$ has labels that do not repeat,
and the other two factors share one of the labels.
The sums over the $s$ and $S$ with the two non-repeating labels
yield $Z^{(1)}_c=O(2d)^{-1}$.
The sums over the remaining labels
are bounded above by $Z^{(2)}_c=O(2d)^{-2}$.
It is then straightforward to verify
that $\Gamma^{(3,5)}_c \le O(2d)^{-3}$.

For $\Gamma^{(3,6)}_c$, the six sums over $s$ give $(Z^{(1)})^3=O(2d)^{-3}$,
and this leads to
$\Gamma^{(3,6)}_c \le O(2d)^{-3}$.
This completes the proof of \eqref{e:Gamma3bd}.

\smallskip \noindent \emph{Proof of \eqref{e:Gamma4bd}.}
We use the bound $\abs{{\cal I}_{rs}}\le1$ in
\eqref{e:Def_Jcal3} and \eqref{e:Gam4til} to obtain
\begin{align}
    |\tilde\Gamma^{(4)}_c|
    &\le\sum_{m=4}^{\infty}\frac{1}{m!} \sum_{s_1,\dots,s_m\in {\cal E}}
    \sum_{S_1\ni s_1} z_c^{|S_1|} \cdots  \sum_{S_m\ni s_m}
    z_c^{|S_m|}
    \sum_{1\le i<j \le m}
    \sum_{ (k,l)\in A_{ij} }  \, \sum_{(p,q)\in A_{kl}} \, \sum_{(r,s)\in A_{pq}}
    \V_{ij} \V_{kl} \V_{pq} \V_{rs}
    .
\label{e:Gam4tilbd}
\end{align}
We denote the cardinality of the label
set by $n=\abs{ \set{i,j,k,l,p,q,r,s} }$, so $n \in \{4,5,6,7, 8\}$,
and exchange the sums over vertices and labels.
As in \eqref{e:Gamma34}, this allows us to rewrite the upper bound
of \eqref{e:Gam4tilbd} in the form
\begin{align}
    |\tilde\Gamma^{(4)}_c|
    &\le
    \Gamma^{(0)}_c
    \sum_{n=4}^8
    \sum_i
    \alpha_{n,i} Z^{(4,n,i)}_c,
\end{align}
where the sum over $i$ is a finite sum, the $\alpha_{n,i}$ are constants
whose values are immaterial, and each $Z^{(4,n,i)}$ is of the
form
\begin{align}
    Z^{(4,n,i)}_c
    & =
    \sum_{s_1,\dots,s_n\in {\cal E}}
    \sum_{S_1\ni s_1} z_c^{|S_1|} \cdots  \sum_{S_n\ni s_n}
    z_c^{|S_n|}
    \V^{(n,i)}
    ,
\end{align}
with $\V^{(n,i)}$ a product of $4$ factors of $\V_{ab}$
having $n$ distinct labels in all.
Since $\Gamma^{(0)}_c=O(1)$ (as observed below \eqref{e:Gamma34}),
it suffices to show that each $Z^{(4,n,i)}_c$ is $O(2d)^{-3}$.

For $n=4,5$ or 6,
we substitute one of the factors in
$\abs{ \V_{ij} \V_{kl} \V_{pq} \V_{rs} }$ by 1,
with the restriction that the remaining three factors have at least
four different labels.
The sums involving the replaced factor yield 1 or $2dr_c$ or $(2dr_c)^2$
depending on whether this factor has 0,1 or 2 distinct labels
from the remaining three factors;
all three cases are $O(1)$ by \eqref{e:rcbd}.
The sums involving the other
three factors reduce to the cases $\Gamma^{(3,4)}_c$,
$\Gamma^{(3,5)}_c$ or $\Gamma^{(3,6)}_c$
which by \eqref{e:Gamma3bd} are $O(2d)^{-3}$.

For $n=7$ or 8,
we consider three factors in the product $\abs{ \V_{ij} \V_{kl} \V_{pq} \V_{rs} }$
that have six different labels and bound the fourth factor by 1.
The sums involving the fourth factor yield $2dr_c$ and $(2dr_c)^2$
for $n=7$ and $n=8$, respectively.
By \eqref{e:rcbd}, in both cases the contribution is $O(1)$.
By the bound on $\Gamma^{(3,6)}_c$ of \eqref{e:Gamma3bd},
the sums involving the six distinct labels is $O(2d)^{-3}$.
This completes the proof of \eqref{e:Gamma4bd} and of the lemma.
\end{proof}

\subsection{Estimates for lace expansion}
\label{sec:Qestimates}

Throughout this section, we assume the leading behaviour \eqref{e:zcleading}
(proved in the present paper in Lemma~\ref{lem:zcleading})
but do not make use of higher order asymptotics.
We prove three lemmas that were used in
Sections~\ref{sec:term2}--\ref{sec:term3}.
Recall from \eqref{e:Qdef} the definition
\begin{equation}
\label{e:Qdef-bis}
    Q(x)=
    \sum_{C_0\ni 0}   \sum_{C_x \ni x}
    z_c^{|C_0|+|C_x|} \1_{C_0 \cap C_x \neq \varnothing}
    .
\end{equation}
The following lemma gives a good estimate for $Q(s)$ when $\|s\|_1=1$,
and gives a crude (but sufficient) estimate for $\|x\|_1>1$.

\begin{lemma}
\label{lem:Q}
For lattice trees or lattice animals, and for a neighbour $s$ of the origin,
\begin{equation}
\label{e:Q_2eq}
    Q(s)
    =  2z_cg_c^3  -  \frac{ \ee^2 }{(2d)^2}  +  o(2d)^{-2}
    = \frac{2\ee^2}{2d} + o(2d)^{-1}.
\end{equation}
In addition, for any $x$, $Q(x) \le O(2d)^{-\frac 12 \|x\|_1}$.
\end{lemma}

\begin{proof}
In \eqref{e:Qdef-bis},
the clusters $C_0$ and $C_x$ only
contribute to the sum in $Q(x)$ if they have a vertex in common,
say $y$. There is a path connecting 0 and $y$ contained in $C_0$,
and a path connecting $y$ and $x$ contained in $C_x$,
and we can choose these paths to intersect only at $y$.
We
denote the paths by $\omega^0$ and $\omega^x$, respectively.
The union of $\omega^0$ and $\omega^x$ forms a path connecting
0 to $x$ and passing through $y$, which we call $\omega$.
It has length at least $\|x\|_1$, and this leads to the upper bound
$Q(x) \le S_{z_c}^{(\|x\|_1,2)}$.  Together with Lemma~\ref{lem:DmGn},
this proves that $Q(x) \le O(2d)^{-\frac 12 \|x\|_1}$.

It remains to prove the first equality of \eqref{e:Q_2eq},
as the second equality then follows
immediately from Lemma~\ref{lem:zcleading}.
We write $Q^{n}(s)$ to refer to the contribution to $Q(s)$
due to configurations where there exists such a path  $\omega$ of
length $n$ (the union of $\omega^0$ and $\omega^s$ as in the
previous paragraph) and no shorter path.  Since
\begin{equation}
\label{e:Q5bd}
    Q^{\ge 5}(s) \le S_{z_c}^{(5,2)} \le O(2d)^{-5/2},
\end{equation}
we can
restrict attention to $Q^n$ for $n \le 4$.
For the case of lattice animals,
the contributions in which $C_0$ or $C_s$ has a cycle containing
both  0 and $s$ is easily seen to be $o(2d)^{-2}$.
Therefore, we assume henceforth that each of $C_0$ and $C_s$ does not have
a cycle that contains both $0$ and $s$.

For $Q^3$, we have
$\omega=(0,s',s'+s,s)$ for some neighbour $s'$ of the origin perpendicular to $s$.
There are $2d-2$ such paths and each of them has
four possibilities for $y$. If we treat the clusters
attached to the vertices in $\omega^0$ and $\omega^s$
as five independent clusters, we obtain the upper bound
\begin{equation}
    Q^{3}(s)  \leq    4(2d-2)\z{3}^3 g_c^5
    =    \frac{4\ee^2}{(2d)^2}+ o(2d)^{-2},
\end{equation}
with the last equality due to Lemma~\ref{lem:zcleading}.
For a lower bound, we use inclusion-exclusion and subtract
from the upper bound the contribution when there are pairwise
intersections among the ribs that belong to the same path,
either $\omega^0$ or $\omega^s$. This gives
\begin{equation}
    Q^{3}(s)  \geq    4(2d-2)\z{3}^3g_c^3
    \sqb{  g_c^2 - 4Q(s) - 2Q(s'+s)   }
    =    \frac{4\ee^2}{(2d)^2}+ o(2d)^{-2}
\end{equation}
(subtraction of $Q(s)$ in the middle expression also accounts
for configurations which are counted by $Q^1(s)$ rather than $Q^3(s)$).
We conclude that
\begin{equation}
    Q^{3}(s)   = \frac{4\ee^2}{(2d)^2}
    +    o(2d)^{-2}.
\end{equation}

For $Q^1$, the path $\omega$ is given by $\omega=(0,s)$.
This means that the bond $\{0,s\}$ is contained in either
$C_0$ or $C_s$, say in $C_0$. In this case, $C_0$ consists
of the edge $\{0,s\}$ and two nonintersecting subclusters,
$C^*_0$ and $C^*_s$, the first one attached at 0 and the
second at $s$.
Let $\U^*_{01}=-1$ if the subclusters $C^*_0$ and $C^*_s$
have a common vertex, and 0 otherwise.
Exchanging the roles of
$C_0$ and $C_s$, and subtracting the contribution due to
the event in which both clusters $C_0$ and $C_s$ contain
the bond $\{0,s\}$, yields
\begin{align}
    Q^1(s)
    & =
    2\z{3}g_c
    \sum_{C^*_0\ni0}\sum_{C^*_s\ni s}
    \z{3}^{|C^*_0| + |C^*_s| }\ob{1+ \U^*_{01}}
    -  \z{3}^2 \bigg [  \sum_{C^*_0\ni0}\sum_{C^*_s\ni s}
    \z{3}^{|C^*_0| + |C^*_s|} \ob{1+ \U^*_{01}}\bigg ]^2
    \nonumber \\
    & =
    2\z{3} g_c ^3 - 2\z{3} g_c Q(s)
    -  \z{3}^2 \sqb{  g_c^2 -  Q(s) }^2
    .
\end{align}

Since $z_c^2 Q(s)=o(2d)^{-2}$,
together with the contributions analysed previously this gives
\begin{align}
    Q(s)
    & =
    2\z{3} g_c ^3 - 2\z{3} g_c Q(s)
    -  \z{3}^2  g_c^4
    + \frac{4\ee^2}{(2d)^2}
    +    o(2d)^{-2}
    .
\end{align}
We conclude from this that
\begin{equation}
\label{e:Qproduct}
    (1+2z_cg_c)Q(s) = 2\z{3} g_c ^3 + \frac{3\ee^2}{(2d)^2} + o(2d)^{-2}.
\end{equation}
The factor multiplying $Q(s)$ is equal to $1+2(2d)^{-1}+o(2d)^{-1}$,
so we obtain $Q(s)$ by multiplying the right-hand side
of \eqref{e:Qproduct} by
$1-2(2d)^{-1}+o(2d)^{-1}$.  This yields
the first equality of \eqref{e:Q_2eq} and completes the proof.
\end{proof}

The next lemma is applied in
Lemmas~\ref{lem:Pi1-2} and \ref{lem:Pi2-bis}.
For a neighbour $s$ of the origin, we define
\begin{align}
\label{e:Qstardef}
    Q^*(s)
    &
    = \sum_{C_0\ni0}
    \sum_{C_1\ni s }
    \sum_{C_2\ni0}
    \z{3}^{\abs{R_0}+\abs{R_1}+\abs{R_2}} \U_{01}\U_{12}
    .
\end{align}

\begin{lemma}
\label{lem:QPi1}
For lattice trees or lattice animals, and for a neighbour $s$ of the origin,
\begin{align}
\label{e:R0123}
    Q^*(s)
    &
    = \frac{\ee^3}{2d} +  o(2d)^{-1},
\end{align}
\end{lemma}

\begin{proof}
It is straightforward to verify that the contribution when $C_1$
contains a cycle containing $0$ and $s$ produces an error term,
so we assume that there is no such cycle.
If $C_1$
contains the bond $(0,s)$, then
$\U_{01}\U_{12}=1$.
In this case, we can regard $C_1$ as consisting
of the edge $(0,s)$ and two non-intersecting
clusters $C_1^0$ and $C_1^1$ attached at 0 and $s$, respectively.
Let $\U^*_{01}=-1$ if  $C_1^0$ and $C_1^1$ have a
common vertex, and 0 otherwise.
We obtain
\ecus{eq:sumLT_3}{
    Q^*(s)
    &= \z{3}g_c^2 \sum_{R_1^0\ni 0,  R_1^1\ni e_1}
    \z{3}^{|R_1^0|+|R_1^1|}\ob{1+\U^*_{01}}
    + \!\!\!\! \sum_{ \substack{R_0,R_2\ni0\\ R_1\ni e_1, R_1\not \ni(0,e_1)} }
    \z{3}^{\abs{R_0}+\abs{R_1}+\abs{R_2}}\U_{01}\U_{12}
    +o(2d)^{-1}.
}
Arguments of the type used several times previously show that
the second sum on the right-hand side is $o(2d)^{-1}$.
Therefore,
\begin{align}
\label{e:R0123-bis}
    Q^*(s)
    &= \z{3}g_c^4 - \z{3}g_c^2 Q(s) +o(2d)^{-1}
    = \frac{\ee^3}{2d} +  o(2d)^{-1},
\end{align}
where the second equality is due to Lemmas~\ref{lem:zcleading}
and \ref{lem:Q}.
\end{proof}

Finally, we prove the following lemma, which
is a restatement of Lemma~\ref{lem:Pi2}.  It provides an important
ingredient in the proof of Theorem~\ref{thm:Pi-2terms-bis}.

\begin{lemma}
\label{lem:Pi2-bis}
For lattice trees or lattice animals,
\begin{align}
\label{e:Pi2-2-bis}
    \hat\Pi^{(2)}_{z_c}
    & =  \frac{11\ee}{(2d)^2}  + o(2d)^{-2}.
\end{align}
\end{lemma}

\begin{proof}
We give the proof only for the case of
lattice trees.  With minor changes, the proof extends to lattice animals.
Recall from \eqref{eq:Pi2} that
\ecu{eq:Pi2-bis}{
    \hat\Pi^{(2)}_{z_c}
    =
    \sum_{x \in \Zd}
    \Pi^{(2)}_{z_c}(x)   =
    \sum_{x \in \Zd}
    \sum_{\omega\in\mathcal W(x)\atop |\omega|\ge 2}
    \z{3}^{|\omega|}
    \sqb{ \prod_{i=0}^{|\omega|} \sum_{R_i\ni \omega(i)}\z{3}^{|R_i|}}
    \sqb{   \sum_{L \in\mathcal L ^{(2)}[0,|\omega|]} \,
    \prod_{ij\in L}\U_{ij}
    \prod_{i'j'\in \mathcal C(L)}\ob{1+ \U_{i'j'} }   },
}
where the set of laces is
\begin{equation}
    \mathcal L^{(2)}[0,|\omega|]
    =
    \big\{ \set{0j,j|\omega|} : 0<j<|\omega| \big\}
    \cup
    \big\{\set{0j,i|\omega|} : 0<i<j<|\omega| \big\},
\end{equation}
and where the set $\mathcal{C}(L)$ compatible with $L$ is defined
below \eqref{eq:Pi2}.
Let $\Pi^{(2,n)}_{z_c}(x)$ denote the contribution to
$\Pi^{(2)}_{z_c}(x)$ due to $|\omega|=n$ on the right-hand side of
\eqref{eq:Pi2-bis}.
We will show that
\begin{align}
\label{e:Pi42}
    \hat \Pi^{(2,2)}_{z_c}
    &  =
    \frac{5\ee}{(2d)^2}  +  o(2d)^{-2},
    \\
\label{e:Pi43}
    \hat \Pi^{(2,3)}_{z_c}
    &  =   \frac{5\ee}{(2d)^2}  +  o(2d)^{-2},
        \\
\label{e:Pi44}
    \hat \Pi^{(2,4)}_{z_c}
    &  =   \frac{\ee}{(2d)^2}  +  o(2d)^{-2},
    \\
\label{e:Pi45}
    \hat \Pi^{(2,>4)}_{z_c}
    &  =     o(2d)^{-2},
\end{align}
which proves \eqref{e:Pi2-2-bis}.

Before entering into the details, we recall diagrammatic
estimates for lattice trees that have
been developed and discussed at length in
\cite{Hara08,HS90b,Slad06} (for lattice animals the best reference
is \cite{Hara08}).
These techniques are based on the diagrams in Figure~\ref{fig:bdPi2},
which inspire the upper bound
\begin{equation}
\label{e:Pi2ST}
    \hat\Pi^{(2)}_{z_c} \le 2 S_{z_c}^{(1,4)} S_{z_c}^{(1,3)}.
\end{equation}
Here the occurrence of $S^{(1,n)}$ on the right-hand side is
connected with the fact that each loop in the bounds on diagrams
in Figure~\ref{fig:bdPi2} must
consist of at least one bond, while the appearance of $3$ and $4$ is due to the
$7$ lines in the adjacent squares, each of which represents a two-point
function.  When we consider configurations for which it is guaranteed
that those two-point functions must take at least $k$ steps in total,
the upper bound \eqref{e:Pi2ST} can be improved to an upper bound
\begin{equation}
\label{e:Pi2ST-k}
    2\sum_{i+j=k} S_{z_c}^{(i,4)}S_{z_c}^{(j,3)} = O(2d)^{-k/2},
\end{equation}
and once $k =5$ this is an error term.  We will exploit this principle
in the following, beginning with \eqref{e:Pi45} for its simplest
illustration.

\begin{figure}[!h]
 \centering
 \includegraphics[scale=2.8]{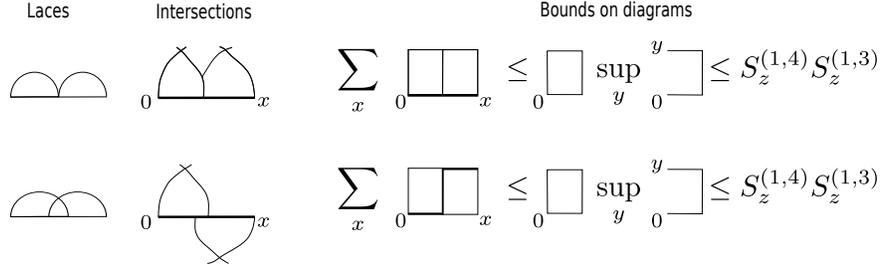}
 \caption{\label{fig:bdPi2} The two generic laces consisting of two bonds,
  schematic diagrams showing the corresponding rib intersections for a nonzero
  contribution to $\Pi^{(2)} (x)$, and diagrammatic bounds for
  the contributions to $\Pi^{(2)}(x)$.
  Diagram lines corresponding to the backbone
  joining $0$ and $x$ are shown in bold. }
\end{figure}

\smallskip \noindent \emph{Proof of \eqref{e:Pi45}.}
When $\omega$ has length at least $5$, then from
\eqref{e:Pi2ST-k} we immediately obtain
\begin{equation}
    \hat\Pi^{(2,>4)}_{z_c} \le
    2\sum_{i+j = 5} S_{z_c}^{(i,3)} S_{z_c}^{(j,4)}
    \le O(2d)^{-5/2},
\end{equation}
which gives \eqref{e:Pi45}.

\smallskip \noindent \emph{Proof of \eqref{e:Pi42}.}
When $|\omega|=2$,
there is only the lace $L=\set{01,12}$, and $\mathcal{C}(L)=\varnothing$.
Therefore,
\ecu{eq:Pi2_2}{
    \hat \Pi^{(2,2)}_{z_c}
    = \sum_{x: \|x\|_1 \in \{0,2\}}
    \sum_{\omega\in\mathcal W(x)\atop |\omega|= 2}
    \z{3}^{2}   \sum_{R_0\ni\omega(0),R_1\ni\omega(1)\atop R_2\ni\omega(2)}
    \z{3}^{|R_0| + |R_1| + |R_2|} \U_{01}\U_{12}.
}
For $x=0$,
we have $\omega=(0,s,0)$,
where $s$ is a neighbour of the origin, and Lemma~\ref{lem:QPi1} gives
\begin{align}
\label{e:Pi22}
    \Pi^{(2,2)}_{z_c}(0)
    &
    =   2d z_c^2 \left[ \frac{\ee^3}{2d} + o(2d)^{-1} \right]
    =   \frac{\ee}{(2d)^{2}} + o(2d)^{-2}.
\end{align}

When $\|x\|_1=2$,
one way to achieve $ \U_{01}\U_{12}=1$ is to have
either $R_0$ or $R_1$ contain the bond $(0,s)$,
and either $R_1$ or $R_2$ contain the bond $(s,x)$.
To obtain a lower bound from such configurations,
we treat the subribs emanating from
these bonds as independent and use inclusion-exclusion to
subtract the possible intersections among them.  This yields
\begin{align}
\label{e:Pi22-lbd}
    \sum_{x:\|x\|_1=2} \Pi^{(2,2)}_{z_c}(x)
    &\geq   4 (2d) (2d-1) \z{3}^4 g_c
    \sqb{ g_c^4 - 2Q(s)\z{3}g_c^2 }
     =    \frac{4\ee}{(2d)^2}
    +  o(2d)^{-2}.
\end{align}
If  $(0,s)$ is not present in $R_0$ and $R_1$,
or $(s,x)$ is not present in $R_1$ and $R_2$,
then an intersection among the corresponding ribs requires
at least four edges (including the step in $\omega$),
so as in Figure~\ref{fig:bdPi2} we obtain for
this case the crude upper bound $S_{z_c}^{(4,4)}S_{z_c}^{(1,3)} + S_{z_c}^{(1,4)}S_{z_c}^{(4,3)}$.
This implies
\begin{align}
    \sum_{x:\|x\|_1 =2} \Pi^{(2,2)}_{z_c}(x)
    &\leq   4 (2d) (2d-1) \z{3}^4 g_c^5
    +    S_{z_c}^{(4,4)}S_{z_c}^{(1,3)} + S_{z_c}^{(1,4)}S_{z_c}^{(4,3)}
    =    \frac{4\ee}{(2d)^2}
     +  o(2d)^{-2},
\end{align}
and, with \eqref{e:Pi22}--\eqref{e:Pi22-lbd},
this completes the proof of \eqref{e:Pi42}.

\smallskip \noindent \emph{Proof of \eqref{e:Pi43}.}
When $|\omega|=3$, there are three laces:
$L=\set{01,13}$, $L=\set{02,23}$, $L=\set{02,13}$.

\smallskip \emph{The laces $L=\set{01,13}$, $L=\set{02,23}$.}
By symmetry, both laces
give the same contribution to \eqref{eq:Pi2-bis},
so of these we only study the contribution
due to $L=\set{01,13}$ (with $\mathcal{C}(L)=\{12,23\}$), which is
\ecu{eq:Pi2_3a}{
    \sum_{x: \|x\|_1 \in \{1,3\}}
    \sum_{\omega\in\mathcal W(x)\atop |\omega|= 3}
    \z{3}^{3}
    \sum_{R_0\ni\omega(0),R_1\ni\omega(1)\atop R_2\ni\omega(2),
    R_3\ni\omega(3)}
    \z{3}^{|R_0| + |R_1| + |R_2| + |R_3|}
    \U_{01}\U_{13}\ob{ 1 + \U_{12} }\ob{ 1 + \U_{23} }.
}

\smallskip \emph{Case of $\|x\|_1=1$.}
When $\|x\|_1=1$, $\omega$ either has the form $\omega=(0,x,y,x)$
for $y$ a neighbour of $x$ (possibly $y=0$), or
$\omega=(0,s,s+y,x)$ for a neighbour $s$ of the
origin distinct from $x$ and for $y\in \{-s,x\}$.
In the first case, when $\omega=(0,x,y,x)$,
we have $\U_{13}=-1$ since $\omega(1)=x=\omega(3)$.
Using $\ob{ 1 + \U_{12} }\ob{ 1 + \U_{23} }\leq 1$,
Lemmas~\ref{lem:zcleading} and \ref{lem:Q}, we find that
this contribution to \eqref{eq:Pi2_3a} is bounded above by
\begin{align}
        &  (2d)^2  z_c^{3}g_c^2 Q(x)
        = \frac{2\ee}{(2d)^2} +  o(2d)^{-2}.
\end{align}
Also, using
$(-\U_{01})\ob{ 1 + \U_{12} }\ob{ 1 + \U_{23} }
\ge (-\U_{01}) \ob{ 1 + \U_{12} + \U_{23} }$, this
contribution to \eqref{eq:Pi2_3a}
is bounded below by
\begin{align}
    &
    (2d)^2  \z{3}^{3}
    \sqb{  g_c^2 Q(x) -  O(2d)^{-2} - Q(x)^2 }
    =   \frac{2\ee}{(2d)^2} +  o(2d)^{-2},
\end{align}
where we omit the straightforward details for
the $\U_{01}\U_{12}$ term.
This contribution gets counted twice to account also
for the lace $L=\set{02,23}$.

In the second case, when $\omega=(0,s,s+x,x)$, the
contribution to \eqref{eq:Pi2_3a}
is bounded above by
\begin{align}
      (2d)^2  \z{3}^{3} g_c^2
      \sum_{R_1\ni s,R_3\ni x}\z{3}^{|R_1| + |R_3|} (-\U_{13})
      &=  (2d)^2  \z{3}^{3} g_c^2  Q(x-s)
      =  o(2d)^{-2}
      ,
\end{align}
where we have employed the straightforward improvement
$Q(x-s) \le O(2d)^{-2}$ to the crude bound of Lemma~\ref{lem:Q},
for $\|x-s\|_1=2$.  Also, when $\omega=(0,s,0,x)$, it can be checked
that due to the factor $(1+\U_{12})$ at least 6 bonds are
required to accomplish the intersections required for $\U_{01}\U_{13}=1$.
Therefore, using the upper bound \eqref{e:Pi2ST-k}, this contribution
is at most $O(2d)^{-3}$.

\smallskip \emph{Case of $\|x\|_1=3$.}
When $\|x\|_1 =3$, it can be checked that the required intersections
cannot be accomplished without using at least 5 bonds, and we conclude
from the upper bound \eqref{e:Pi2ST-k} that the total contribution from
all such $x$ is at most $O(2d)^{-5/2}$.

\smallskip \emph{The lace $L=\set{02,13}$.}
Its contribution to \eqref{eq:Pi2-bis} is
\ecu{eq:Pi2_3b}{
    \sum_{x: \|x\|_1 \in \{1,3\}}
    \sum_{\omega\in\mathcal W(x)\atop |\omega|= 3}
    \z{3}^{3}   \sum_{R_0\ni\omega(0),\dots,\atop R_3\ni\omega(3)}
    \z{3}^{|R_0| + \dots + |R_3|} \U_{02}\U_{13}\ob{ 1 + \U_{01} }
    \ob{ 1 + \U_{12} }\ob{ 1 + \U_{23} }.
}
When $\|x\|_1=1$,
either $\omega=(0,x,0,x)$, or
$\omega=(0,s,s+y,x)$ for a neighbour $s$ of the
origin distinct from $x$ and for $y\in \{-s,x\}$.
In the first case, automatically $ \U_{0,2}\U_{13}=1$
since $\omega(0)=0=\omega(2)$
and $\omega(1)=x=\omega(3)$.
The contribution to \eqref{eq:Pi2_3b}
is bounded above by
\begin{equation}
    2d \z{3}^{3}   g_c^4
    =    \frac{\ee}{(2d)^2} +  o(2d)^{-2}.
\end{equation}
A matching lower bound is given by
\begin{align}
    2d  \z{3}^{3}
    \sum_{R_0\ni0,\dots,\atop R_3\ni x}\z{3}^{|R_0| + \dots + |R_3|}
    \ob{ 1 + \U_{01} + \U_{12} + \U_{23} }
    &\geq   2d  \z{3}^{3} \sqb{ g_c^4 - 3g_c^2 Q(x) }
    = \frac{\ee}{(2d)^2} +  o(2d)^{-2}.
\end{align}
In the second case, when $\omega=(0,s,s+y,x)$, the
contribution to \eqref{eq:Pi2_3b}
is bounded above by
\begin{align}
    2d(2d-1)  \z{3}^{3}   g_c^2
    \sum_{R_1\ni s, R_3\ni x}\z{3}^{|R_1| + |R_3|} (-\U_{13})
    &\leq 2d(2d-1)  \z{3}^{3}   g_c^2 Q(x-s)
    =  o(2d)^{-2}.
\end{align}
If $\|x\|_1=3$, the lace $L=\set{02,13}$ forces an
intersection between the ribs $R_1$ and $R_3 $,
without intersecting $R_2$ (due to $12,23 \in \mathcal{C}(L)$).
It can be argued that the contribution in
this case is $o(2d)^{-2}$.

\smallskip \noindent \emph{Proof of \eqref{e:Pi44}.}
For $|\omega|=4$,
we first consider the lace $L=\{02,24\}$ with $x=0$, which
is the only case that contributes.  After discussing this
case in detail, we will argue that all remaining contribution
belong to the error term.

For $L=\set{02,24}$ and $x=0$,
the significant walks are
$\omega=(0,s,0,s',0)$ with $s,s'$
neighbours of the origin. There are $(2d)^2$ such walks
and they have $\U_{02} \U_{24}=1$.
Treating the five ribs emanating from the walks as independent,
we obtain the upper bound
\begin{equation}
    (2d)^2\z{3}^4g_c^5
    =   \frac{\ee}{(2d)^2} + o(2d)^{-2},
\end{equation}
and it is straightforward to verify that this is also a lower bound.
This gives the formula $\ee(2d)^{-2} + o(2d)^{-2}$
that we seek for  $\hat\Pi^{(2,4)}$, so it remains to prove that
the remaining terms contribute $o(2d)^{-2}$.

The other walks of length four with $x=0$ form unit squares
containing the origin (the walk $(0,s,s+s',s,0)$ does not contribute
since it has $1+\U_{13}=0$).
If we bound $\U_{02}$ by 1, and use
the fact that there are $O (2d)^2$ such squares, then we
find that this
contribution to $\Pi^{(2,4)}_{z_c}(0)$ is bounded by (with $s,s'$ orthogonal)
\begin{equation}
    O (2d)^2 \z{3}^4g_c^3 Q(s+s')
    =   o(2d)^{-2},
\end{equation}
where $Q(s+s')$ takes into account
the intersection of $R_2$ and $R_4$ forced by
$\U_{24}$.

For the remaining case $x \neq 0$ for $L=\set{02,24}$,
and for all other laces occurring for $|\omega|=4$,
it can be checked that
there must be at least one additional bond, besides the backbone,
in order to create the intersections for a nonzero contribution.  Then
\eqref{e:Pi2ST-k} gives an upper bound
\begin{equation}
    \sum_{i+j=5} S_{z_c}^{(i,4)}S_{z_c}^{(j,3)} = O(2d)^{-5/2}
\end{equation}
for these contributions, which thus belong to the error term.
This completes the proof.
\end{proof}

\section*{Acknowledgements}

The work of YMM was supported in part by CONACYT of Mexico.
The work of GS was supported in part by NSERC of
Canada.

\end{document}